%% file: main.tex
\newif\ifMariaHasTheToken
\MariaHasTheTokenfalse                                

\newcommand{\MariaIsEditing}[1]{\ifMariaHasTheToken #1\fi}

\documentclass[12pt]{amsart}
\usepackage[utf8]{inputenc}
\usepackage{thm-restate}

\usepackage{enumerate,amsthm,todonotes, amssymb,xcolor,comment,subcaption}
\usepackage[footskip=1cm, headheight = 16pt, top=3.7cm, bottom=3cm,  right=2.5cm,  left=2.5cm ]{geometry}
\usepackage[colorlinks,linkcolor=blue,citecolor=red]{hyperref}
\usepackage{geometry}
\newtheorem{question}{Question}[section]
\newtheorem{lemma}[question]{Lemma}
\newtheorem{theorem}[question]{Theorem}
\newtheorem{conjecture}[question]{Conjecture}
\newtheorem{corollary}[question]{Corollary}

\newcounter{tbox}
\newcommand{\sta}[1]{\vspace*{0.3cm}\refstepcounter{tbox}\noindent{ \parbox{\textwidth}{(\thetbox) \emph{#1}}}\vspace*{0.3cm}}

\usepackage[T1]{fontenc}
\usepackage[leqno]{amsmath}
\usepackage{todonotes}
\usepackage[shortlabels]{enumitem}
\usepackage{latexsym}
\usepackage{amsfonts}
\usepackage{tikz}
\usepackage{float}
\usepackage{lmodern}
\usepackage[foot]{amsaddr}
\usepackage{latexsym}
\usepackage{tikzit}
\input{graph.tikzstyles}
\input{graph.tikzdefs}
\usepackage{algorithm}%
\usepackage{algorithmicx}%
\usepackage{algpseudocode}%
\usepackage{comment}
\usepackage{physics}
\usepackage{graphicx}
\usepackage{mathtools}
\usepackage{amssymb}
\usepackage{bm}
\usepackage[capitalise]{cleveref}
\usepackage{amsthm}
\usepackage{empheq}
\usepackage{amsmath}

\makeatletter
\newcommand{\leqnomode}{\tagsleft@true}
\newcommand{\reqnomode}{\tagsleft@false}
\makeatother

\def\dd{\hbox{-}}

\DeclareMathOperator{\ta}{tree-\alpha}
\DeclareMathOperator{\atw}{tree-\alpha}

\input{commands}

\newcommand{\Mca}{\mathcal{M}}
\newcommand{\Ic}{\mathcal{I}}

\newcommand{\tw}{\mathsf{tw}}

\newcommand{\class}{\mathcal{H}_t}

\newcommand{\inc}{\text{inc}}

\theoremstyle{definition}

\tikzset{every picture/.style={line width=0.75pt}} 

\title{Tree independence number VI. Thetas and pyramids.}

\author{Maria Chudnovsky$^{\dagger}$}
\author{Julien Codsi $^{\mathsection}$}
\thanks{$^{\dagger}$ Princeton University, Princeton, NJ, USA. Supported by NSF Grant
 DMS-2348219, NSF Grant CCF-2505100, AFOSR grant FA9550-22-1-0083 and a Guggenheim Fellowship.}
\thanks{$^{\mathsection}$ Princeton University, Princeton, NJ, USA. Supported by NSF Grant DMS-2348219, by AFOSR Grant FA9550-22-1-0083, and FRQNT Grant 321124.}
\date{}

\begin{document}

\begin{abstract}
Given a family $\mathcal{H}$ of graphs, we say that a graph $G$ is $\mathcal{H}$-free if no induced subgraph of $G$ is isomorphic to a member of $\mathcal{H}$. Let $W_{t\times t}$ be the $t$-by-$t$ hexagonal grid and let $\mathcal{L}_t$ be the family of all graphs $G$ such that $G$ is the line graph of some subdivision of $W_{t \times t}$. We denote by $\omega(G)$ the size of the largest clique in $G$. We prove that for every integer $t$ there exist  integers $c_1(t)$, 
$c_2(t)$ and $d(t)$ such that 
every (pyramid, theta, $\mathcal{L}_t$)-free graph $G$ satisfies:
\begin{itemize}
\item $G$ has a tree decomposition where every bag has size at most
$\omega(G)^{c_1(t)} \log (|V(G)|)$.
\item If $G$ has at least two vertices, then $G$ has a tree decomposition where every bag has independence number at most $\log^{c_2(t)} (|V(G)|)$.
\item For any weight function, $G$ has a balanced separator that is contained in the union of the neighborhoods of at most $d(t)$ vertices.
\end{itemize}
These results qualitatively generalize the main theorems of \cite{tw3} and \cite{TI2}.

Additionally we show that there exist integers $c_3(t), c_4(t)$ such that
for every (theta, pyramid)-free graph $G$ and for every non-adjacent pair
of vertices $a,b \in V(G)$,
\begin{itemize}
\item $a$ can be separated from $b$ by removing at most $w(G)^{c_3(t)}\log(|V(G)|)$ vertices.
\item $a$ can be separated from $b$ by removing a set of vertices with independence number at most $\log^{c_4(t)}(|V(G)|)$.
\end{itemize}
\end{abstract}

\maketitle
\MariaIsEditing{\textcolor{red}{DO NOT MODIFY! Maria has the token!}}
\section{Introduction} \label{sec:intro}
All graphs in this paper are finite and simple, and all logarithms are base $2$.
Let $G = (V(G),E(G))$ be a graph. For a set $X \subseteq V(G),$ we denote by $G[X]$ the subgraph of $G$ induced by $X$, and by $G \setminus X$ the subgraph of $G$ induced by $V(G) \setminus X$. In this paper, we use induced subgraphs and their vertex sets interchangeably.
For graphs $G,H$ we say that $G$ {\em contains $H$} if $H$ is
isomorphic to $G[X]$ for some $X \subseteq V(G)$. In this case, we say that
$X$ {\em is an $H$ in $G$}.  We say that
$G$ is {\em $H$-free} if $G$ does not contain $H$.
For a family $\mathcal{H}$ of graphs, we say that $G$ is
{\em $\mathcal{H}$-free} if $G$ is $H$-free for every $H \in \mathcal{H}$.

The \emph{open neighborhood of $v$}, denoted by $N_G(v)$, is the set of all vertices in $V(G)$ adjacent to $v$. The \emph{closed neighborhood of $v$}, denoted by $N_G[v]$, is $N(v) \cup \{v\}$. Let $X \subseteq V(G)$. The \emph{open neighborhood of $X$}, denoted by $N_G(X)$, is the set of all vertices in $V(G) \setminus X$ with at least one neighbor in $X$. The \emph{closed neighborhood of $X$}, denoted by $N_G[X]$, is $N_G(X) \cup X$. When there is no danger of confusion, we omit the subscript $G$.
Let $Y \subseteq V(G)$ be disjoint from $X$. We say $X$ is \textit{complete} to $Y$ if all possible edges with an end in $X$ and an end in $Y$ are present in $G$, and $X$ is \emph{anticomplete}
to $Y$ if there are no edges between $X$ and $Y$.

For a graph $G$, a \emph{tree decomposition} $(T, \chi)$ of $G$ consists of a tree $T$ and a map $\chi\colon V(T) \to 2^{V(G)}$ with the following properties:
\begin{enumerate}
\itemsep -.2em
    \item For every $v \in V(G)$, there exists $t \in V(T)$ such that $v \in \chi(t)$.

    \item For every $v_1v_2 \in E(G)$, there exists $t \in V(T)$ such that $v_1, v_2 \in \chi(t)$.

    \item For every $v \in V(G)$, the subgraph of $T$ induced by $\{t \in V(T) \mid v \in \chi(t)\}$ is connected.
\end{enumerate}

For each $t\in V(T)$, we refer to $\chi(t)$ as a \textit{bag of} $(T, \chi)$.  The \emph{width} of a tree decomposition $(T, \chi)$, denoted by $width(T, \chi)$, is $\max_{t \in V(T)} |\chi(t)|-1$. The \emph{treewidth} of $G$, denoted by $\tw(G)$, is the minimum width of a tree decomposition of $G$.  Graphs of bounded treewidth  are well-understood both  structurally
\cite{RS-GMXVI} and algorithmically \cite{Bodlaender1988DynamicTreewidth}.

A {\em clique} in a graph $G$ is a set of pairwise adjacent vertices;
the {\em clique number} of $G$, denoted by $\omega(G)$, is the largest size of a clique in $G$.
A {\em stable (or independent) set} in a graph $G$ is a set of pairwise non-adjacent vertices of $G$. The {\em stability (or independence) number} $\alpha(G)$ of $G$ is the largest size of a stable set in $G$.
Given a graph $G$ with weights on its
vertices, the \textsc{Maximum Weight Independent Set (MWIS)} problem is
the problem of finding a stable set in $G$ of maximum total weight.
This problem is known to be \textsf{NP}-hard \cite{alphahard}, but it can be solved efficiently (in polynomial time)  in graphs of bounded treewidth.

The {\em independence number} of a tree decomposition 
$(T, \chi)$ of $G$ is $\max_{t \in V(T)} \alpha(G[\chi(t)])$. The {\em tree independence number} of $G$, denoted $\ta(G)$, is the minimum independence number of a tree decomposition of $G$. The tree-independence number was defined by Dallard, Milani\v{c} and \v{S}torgel \cite{dms2} as a way to understand graphs whose high treewidth can be explained by the presence of a large clique, and 
targeting the complexity of the \textsc{MWIS} problem.
Combining results of  \cite{dms2} and \cite{dfgkm} yields
an efficient algorithm for the \textsc{MWIS} problem for graphs of bounded $\ta$.  In \cite{lima2024tree}, similar algorithmic results are obtained for a much more general class of problems. 
Recently, the study of the structure of graphs with low $\atw(G)$ has gained momentum, see for example \cite{dms3,twVsClique4,twVsClique5}.

A {\em hole} in a graph is an induced cycle of length at least four.
A {\em path} in a graph is an induced subgraph that is a path. The
{\em length} of a path or a hole is the number of edges in it.
We denote by  $P=p_1 \dd \dots \dd p_k$
be a path in $G$ where $p_ip_j \in E(G)$ if and only if $|j-i|=1$. We say that $p_1$ and $p_k$ are
the {\em ends} of $P$. The {\em interior} of $P$, denoted by
$P^*$, is the set $P \setminus \{p_1,p_k\}$.
For $i,j \in \{1, \dots. k\}$ we denote by $p_i \dd P \dd p_j$ the subpath of
$P$ with ends $p_i,p_j$.

A {\em theta} is a graph consisting of two distinct  vertices $a, b$ and three
paths $P_1, P_2, P_3$ from $a$ to $b$, such that $P_i \cup P_j$ is a hole
for every $i,j \in \{1,2,3\}$. It follows that $a$ is non-adjacent to $b$
and the sets $P_1^*,P_2^*, P_3^*$ are pairwise disjoint and
anticomplete to each other.
If a graph $G$ contains an induced subgraph $H$ that is a theta, and $a, b$ are the two vertices of degree three in $H$, then we say that $G$ contains a
theta \emph{with ends $a$ and $b$}.

A {\em pyramid} is a graph consisting of a vertex $a$ and a triangle $\{b_1, b_2, b_3\}$, and three paths $P_i$ from $a$ to $b_i$ for $1 \leq i \leq 3$,
such that  $P_i \cup P_j$ is a hole for every $i,j \in \{1,2,3\}$.
It follows 
that $P_1 \setminus a, P_2 \setminus a , P_3 \setminus a$ are pairwise disjoint,
and the only edges between them are of the form $b_ib_j$.
It also follows that 
at most one of $P_1, P_2, P_3$ has length exactly one. We say that $a$ is the {\em apex} of the pyramid and that $b_1b_2b_3$ is the {\em base} of the pyramid.

(Theta, pyramid)-free graphs have received significant attention in structural graph theory. Often another family of graphs, called ``prisms'', is excluded.
Treewidth and tree-independence number of (Theta, pyramid, prism)-free graphs
were studied in \cite{tw3} and \cite{TI2}. In \cite{tw3}, a logarithmic upper bound on treewidth is obtained (where the bound depends on the clique number of the graph), while in \cite{TI2}  tree-independence number is bounded by a
polylogarithmic function of the number of vertices.

The scope of this paper is broader in the following sense.
Let $W_{t\times t}$ be the $t$-by-$t$ hexagonal grid, and let
$\mathcal{L}_t$ be the family of all graphs $G$ such that $G$ is the line graph of some subdivision of $W_{t \times t}$. 
Observe that  $\bigcup_{t}\mathcal{L}_t$ contains a sequence of graphs for which the treewidth grows asymptotically as the square root of the number of vertices. It follows that  $\mathcal{L}_t$ needs to be excluded to achieve sub-polynomial (in the number of vertices) bounds on treewidth.
Since all graphs in  $\bigcup_{t}\mathcal{L}_t$ are (pyramid, theta)-free
and have clique number at most three, excluding
 $\bigcup_{t}\mathcal{L}_t$  is necessary even in the class of
 (theta, pyramid)-free graphs with bounded clique number.
 The situation for tree-independence number is similar. 

We prove that this necessary condition is, in fact, sufficient in
(theta pyramid)-free graphs. Since
 every $\mathcal{L}_t$ contains a prism if we choose $t$ large enough,
 this is a qualitative generalization of the results of \cite{tw3} and \cite{TI2} (the degree of the polynomial in $\log (|V(G)|)$ in the bound on the tree-independence number is worse here).
For an integer $t$, let $\mathcal{M}_t$ be the class of all (theta, pyramid, $\mathcal{L}_t$)-free graphs. We prove:
\begin{theorem}
    \label{thm: log tw}
    For every positive integer $t$, there is an integer $c=c(t)$ such that every graph $G\in\mathcal{M}_t$ has treewidth at most $\omega(G)^c\log(|V(G)|)$.
\end{theorem}

\begin{theorem}\label{thm: small tree alpha}
    For every positive integer $t$ there exists $c=c(t)$ such that for every graph $G \in \mathcal{M}_t$ on at least $3$ vertices, we have  $\atw(G) \leq \log^c n$.
    \end{theorem}
We remark that \cref{thm: small tree alpha} follows immediately from
\cref{thm: log tw} using a result of \cite{CESL}; we explain this in Section~\ref{sec: main results}.
\cref{thm: log tw}  is tight since there exist  (theta, triangle)-free graphs with logarithmic treewidth  \cite{mainconj}.  For the same reason
\cref{thm: small tree alpha} is tight up to the degree of the
polynomial (in fact, we do not have a counterexample to $c=1$).

As is explained in \cite{tw3},  by the celebrated Courcelle's theorem \cite{COURCELLE}, \cref{thm: log tw} also implies the existence of polynomial time algorithms for a large class of \textsf{NP}-hard problems, such as \textsc{Stable Set}, \textsc{Vertex Cover}, \textsc{Dominating Set}, and \textsc{Coloring},
when the input restricted to graphs in $\mathcal{M}_t$ with bounded clique
number. Similarly, the algorithmic implications of
\cref{thm: small tree alpha} using results of \cite{lima2024tree} are discussed in \cite{TI2}.

Finally, \cref{thm: small tree alpha} is a promising step toward the following:
\begin{conjecture} [from \cite{TIV}]
    \label{conj:smalltreealph}
    For every positive integer $t$, there is an integer $d=d(t)$ such that for every $n\ge 2$, every
    $n$-vertex graph with no induced minor isomorphic to $K_{t,t}$ or to $W_{t\times t}$ has $\atw$ at most $\log^d n$.
\end{conjecture}

Let $G$ be a graph and let $A,B\subseteq G$ be disjoint. We say that a set $X\subseteq V(G)\setminus(A\cup B)$ \textit{separates} $A$ from $B$ if for every connected component $D$ of $G\sm X$, $D\cap A =\mt$ or $D\cap B = \mt$.
Let $G$ be a graph and let $a,b$ be two non-adjacent vertices of $G$. We say that a set $X\subseteq V(G)\setminus\set{a,b}$ \textit{separates} $a$ from $b$ if for every connected component $D$ of $G\sm X$, $|D\cap\set{a,b}|\leq 1$.
A graph is said to be \textit{$k$-pairwise separable} if for every pair of non-adjacent vertices of $G$, there exists a set $X$ with $|X| \leq k$  that separates them from each other.

In order to prove \cref{thm: log tw}, we need a result
on pairwise-separability. We were able to obtain such a result in a more
general setting (without excluding $\mathcal{L}_t$), which may be of independent interest.
Here we denote by $\class$ the class of ($K_t$, pyramid, theta)-free graphs.
\begin{restatable}{theorem}{pairwisesep}\label{thm: small pairwise separator}
    For every integer $t\geq 2$, there exists a positive integer  $c$ such that every 
    $n$-vertex graph in $\class$ is $t^c \log n$-pairwise separable.
\end{restatable}
Once again, using a  result of \cite{CESL}, we also get a version of this theorem where the size of the separator is replaced by its independence number
(see Section~\ref{sec: main results} for details):
\begin{theorem}\label{thm: small alpha pairwise separator}
   There exists an integer  $c$  such that for every $n$-vertex (theta, pyramid)-free graph $G$ and every non-adjacent pair $u,v \in V(G)$ there exists $X \subseteq V(G)\sm\set{u,v}$ with $\alpha(X) \leq  \log^c n$ such that $X$ separates $u$ from $v$.
\end{theorem}

For a graph $G$, a function  $w\colon V(G) \rightarrow [0,1]$ is a {\em weight function} if $\sum_{v \in V(G)} w(v) \leq 1$.
For  $S \subseteq V(G)$, we write $w(S) \coloneqq  \sum_{v \in S} w(v)$.
A weight function $w$ is a {\em normal weight function} on $G$ if $w(V(G))=1$. If $0<w(V(G))<1$,  we call the function $w'\colon V(G) \rightarrow [0,1]$ given by $w'(v) = \frac{w(v)}{\sum_{u \in V(G)} w(u)}$ the
{\em normalized weight function of $w$}.
Let $c \in [0, 1]$ and let $w$ be a weight function on $G$.
A set $X \subseteq V(G)$ is a {\em $(w,c)$-balanced separator} if $w(D) \leq  c$ for every component $D$ of $G \setminus X$.
The set $X$ is a {\em $w$-balanced separator} if $X$ is a $(w,\frac{1}{2})$-balanced separator.
Given two sets of vertices $X$ and $Y $ of $G$, we say that $X$ is a {\em core}
for $Y$ if $Y \subseteq N[X]$.
A graph $G$ is said to be {\em $k$-breakable} if for every weight function $w\colon V(G) \rightarrow [0,1]$, there exists a $w$-balanced separator with a core of size strictly less than $k$. When the weight function $w$ is clear from the context, we may omit it from the notation.
Our last result is:
\begin{restatable}{theorem}{domsep}
  \label{thm:domsep}
    For every positive integer $t$, there is an integer $d=d(t)$ such that every graph $G \in \mathcal{M}_t$ is $d$-breakable.
\end{restatable}

In Section~\ref{sec: main results},  we follow the outline of the proof (and use some results) of \cite{tw15} to deduce \cref{thm: log tw} from \cref{thm: small pairwise separator} and \cref{thm:domsep}.

\subsection{Proof outline and organization}

Most of the work in this paper is devoted to proving
\cref{thm: small pairwise separator} and  \cref{thm:domsep}. 
\cref{thm: log tw}, \cref{thm: small tree alpha} and \cref{thm: small alpha pairwise separator} are deduced from them using existing results in Section~\ref{sec: main results}.  An important tool for both the main proofs is ``extended strip decompositions'' from \cite{Threeinatree}. They are introduced in
\cref{sec:esd definition}.

Let us start by outlining the proof of Theorem~\ref{thm:domsep} that is proved in  \cref{sec: dbs}. The high-level idea of the proof is similar to
\cite{TIV}, but the details are different because of the different properties of the graph class in question.
Let $G \in \mathcal{M}_t$. We may assume that $G$ is connected.
For a contradiction, we fix a weight function $w$ such that  $G$ does not have a $w$-balanced separator with a small core.
By using the normalized weight function of $w$, we may assume that $w$ is normal. By Lemma~5.3 of \cite{QPTAS}, there is a path $P=p_1 \dd \dots \dd p_k$ in $G$ such that $N[P]$ is a $w$-balanced separator in $G$. We choose  $P$ with $k$
minimum; consequently we  may assume that there is a component $B$ of $G \setminus N[P \setminus \{p_k\}]$
with $w(B) > \frac{1}{2}$. We now analyze the structure of the
set $N=N(B) \subseteq N(P)$. To every vertex $n \in N$ we assign
a subpath $I(n)$ of $P$, that is the minimal subpath of $P$ that contains
$N(n) \cap P$. We define a new graph $H$  with vertex set $N$ where $n_1$ and $n_2$ are adjacent if and only if $I(n_1) \cap I(n_2) \neq \emptyset$.
Let $S$ be a maximum stable set in $H$.
We first show that for every $s \in S$, we can find a small core (in $G$)
for the set $N_H[s]$ (when viewed as a subset of $G$). This, in particular, allows us to assume that $|S|$ is
large. Now, we focus on one vertex  $n \in S$ and use it 
to show that $G$ (with a subset with a small core deleted)  admits an extended strip decomposition. This allows us to produce a separator $X(n)$ with a small core that is not yet balanced, but exhibits several useful properties. More explicitly, the component of
$G \setminus X(n)$ with maximum $w$-weight  only meets $P$ on one side of $I(n)$. So $n$ either ``points left'' or ``points right''.
Then we show that the vertex of $S$ with the earliest neighbors in $P$ points right, and the vertex of $S$  with the latest neighbors in $P$ points left. Now we focus on two vertices  $n,n' \in S$ such that $I(n)$ and $I(n')$ are consecutive along $P$  where the change first occurs, and conclude
that $X(n) \cap X(n')$ is a $w$-balanced separator in $G$.
 This completes the proof of Theorem~\cref{thm:domsep}.

We now turn to the proof of \cref{thm: small pairwise separator}. This is the most novel part of the paper, where several new ideas are introduced. Let $G \in \class$ and let $a,b \in V(G)$ be non-adjacent.  We consider carefully chosen subsets $X_i$ of
balls of radius $i$  around $a$ with $a \in X_i$,  and  iteratively construct a set  $C$
that, when the process stops,  separates $a$ from $b$.
We will show that $|C| \leq t^c  \log (|V(G)|)$ (where the same $c$ works for all graphs in $\class$), thus proving \cref{thm: small pairwise separator}.

Throughout the process, $X_i$ satisfies two key properties. The first one is
that $X_i$ is ``cooperative'' (as defined in \cref{sec: cooperative matroid}).
The second one is that at each step of the construction, we only
add to $X_i$ vertices that have a lower value in the partition of $V(G)$
defined by \cref{degenerate partition} (see \cref{sec: cooperative set evolving degenerate partition} for the defintion of ``value'').

At each step $i$, we examine the attachments $N_i$ of the component $D_i$ of $G \setminus (N[X_i] \cup C)$ with $b \in D$. Note that $N(D_i) \subseteq N(X_i)$. First, we show that $N_i \cap N(b)$ is small and add $N_i \cap N(b)$ to $C$.
Next we  define two matroids on $N_i \setminus N(b)$: $\mathcal{M}_1$ is the  matching matroid into
$X_i$, and
$\mathcal{M}_2$  is the matroid whose independent sets are linkable to $b$ by disjoint paths with interior in $D_i$; both matroids are defined precisely in \cref{sec: cooperative matroid}. Suppose first that there is a large subset $I$ of $N_i$ that is independent in both matroids. We use the fact that $I$ is independent in $\mathcal{M}_1$ and that $X_i$ is cooperative to obtain a large subset $Z$ of $I$ that is  ``constricted'' in $D_i \cup Z$ (see \cref{sec:esd definition} for a precise definition).
This allows us to construct an extended strip decomposition of $(D_i \cup Z, Z)$. Now, we use results from \cref{sec: constricted in ktt free} to get a contradiction to the fact that $Z$ is independent in $\mathcal{M}_2$.

Thus, we may assume that no such set $I$ exists.  We apply the Matroid Intersection Theorem to construct a partition $(A_1,A_2)$ of $N_i \setminus N(b)$ such that the sum of
the ranks $rk_{\mathcal{M}_i}(A_i)$ is small.  By  Menger's Theorem, there is a small subset $Z_2$ of $D_i \cup A_2$ such that $Z_2$ separates $A_2$ from $b$; we
add $Z_2$ to $C$. By K\"{o}nig's Theorem, there is a small subset $Z_1$
of $N_{X_i}(N_i) \cup N_i$ such that every edge between $A_1$ and $X_i$
has an end in $Z_1$. We add $Z_1 \cap N_i$ to $C$  and focus on $Z_1 \cap X_i$.
By \cref{degenerate partition}, the number of vertices with a neighbor in $Z_1 \cup X_i$
whose value is higher  than the maximum value of a vertex in $Z_1 \cup X_i$ 
is bounded by $t^{c'}$ (where the same $c'$ works for every graph in $\class$); we add all such vertices to $C$. Note that at this point
$N_i \subseteq N(Z_1) \cup C$.
If $N(Z_1) \not \subseteq C$, 
we construct $X_{i+1}=X_i \cup (N(Z_1) \setminus C)$ and continue.

By  \cref{degenerate partition} for   some $i  \leq \log(|V(G)|)$
it holds that $N(Z_1) \subseteq C$.
Now $N_i \subseteq C$  and consequently  
$C$ separates $X_i$ from $b$. Since  $a \in X_i$,  $C$  has the required
separation properties, and we stop the process. This completes the description
of the proof of \cref{thm: small pairwise separator}.

This paper is organized as follows.
In \cref{sec:esd definition}, we define constricted sets and extended strip decompositions. In \cref{sec: dbs}, we prove \cref{thm:domsep}. In \cref{sec: constricted in ktt free}, we establish an important property of constricted sets in $K_{t,t}$-free graphs (a super-class of $\class$),
that will be used in the proof of \cref{thm: small pairwise separator} to
obtain a bound on the size of a set that is independent in both
$\mathcal{M}_1$ and $\mathcal{M}_2$.
In \cref{sec: cooperative matroid}, we define 
cooperative sets and prove several lemmas about their properties; that is also where we describe the application of the Matroid Intersection Theorem. In \cref{sec: cooperative set evolving degenerate partition}, we prove
\cref{thm: small pairwise separator}.
Finally, in \cref{sec: main results}, we show how to use \cref{thm: small pairwise separator} and \cref{thm:domsep} to prove  \cref{thm: log tw}, \cref{thm: small tree alpha}, and \cref{thm: small alpha pairwise separator}.

\section{Constricted sets and extended strip decompositions} \label{sec:esd definition}

\input{esd_definition}

\section{Dominated Balanced Separators in $\mathcal{M}_t$} \label{sec: dbs}
\MariaIsEditing{\textcolor{red}{DO NOT MODIFY! Maria has the token!}}

We need the following results from \cite{TIV}:

\begin{lemma} \label{lem:bigatom}
There exists an integer $c$ with the following property.
Let $t \geq 2$ be an integer. Let $G$ be an $\mathcal{L}_t$-free graph, and let $w$ be a weight function on $G$. Let $D$ be a component of $G$ with $w(D)>\frac{1}{2}$.
Let $Z \subseteq D$, and let $\eta$ be a faithful extended strip decomposition of $(D,Z)$ with pattern $H$. Assume that $w(A) \leq \frac{1}{2}$
for every atom $A$ of $\eta$. 
Then there exists $Y \subseteq V(G)$
with $|Y|   \leq ct^9\log^ct$, such that $N[Y]$ is a $w$-balanced separator in $G$.
\end{lemma}

\begin{lemma}\label{atomboundary}
Let $G,H$ be graphs, $Z \subseteq V(G)$ with $|Z| \geq 2$,
and let $\eta$ be a faithful  extended strip decomposition of  $(G,Z)$
with pattern $H$. Let $A$ be an atom of $\eta$. Then $\delta(A)$ has a core of
size at most~$3$.
\end{lemma}

We also need the following, which is  an immediate corollary of Lemma 6.8 of \cite{QPTAS}:
\begin{lemma} \label{threepaths}
Let $G,H$ be graphs, $Z \subseteq V(G)$ with $|Z| \geq 3$,
and let $\eta$ be an  extended strip decomposition of $(G,Z)$
with pattern $H$. Let $Q_1,Q_2,Q_3$ be paths in $G$,
pairwise anticomplete to each other, and each with an end in $Z$.
Then for every atom $A$ of $\eta$, at least one of the sets
$N[A] \cap Q_1$, $N[A] \cap Q_2$ and $N[A] \cap Q_3$ is empty.
\end{lemma}

Finally, we need the following result from \cite{prismfree}.
\begin{lemma}\label{lem: minimalconnected}
Let $x_1, x_2, x_3$ be three distinct vertices of a graph $G$. Assume that $H$ is a connected induced subgraph of $G \setminus \{x_1, x_2, x_3\}$ such that $V(H)$ contains at least one neighbor of each of $x_1$, $x_2$, $x_3$, and that $V(H)$ is minimal subject to inclusion. Then, one of the following holds:
\begin{enumerate}[(i)]
\item For some distinct $i,j,k \in  \{1,2,3\}$, there exists $P$ that is either a path from $x_i$ to $x_j$ or a hole containing the edge $x_ix_j$ such that
\begin{itemize}
\item $V(H) = V(P) \setminus \{x_i,x_j\}$; and
\item either $x_k$ has two non-adjacent neighbors in $H$ or $x_k$ has exactly two neighbors in $H$ and its neighbors in $H$ are adjacent.
\end{itemize}

\item There exists a vertex $a \in V(H)$ and three paths $P_1, P_2, P_3$, where $P_i$ is from $a$ to $x_i$, such that 
\begin{itemize}
\item $V(H) = (V(P_1) \cup V(P_2) \cup V(P_3)) \setminus \{x_1, x_2, x_3\}$;  
\item the sets $V(P_1) \setminus \{a\}$, $V(P_2) \setminus \{a\}$ and $V(P_3) \setminus \{a\}$ are pairwise disjoint; and
\item for distinct $i,j \in \{1,2,3\}$, there are no edges between $V(P_i) \setminus \{a\}$ and $V(P_j) \setminus \{a\}$, except possibly $x_ix_j$.
\end{itemize}

\item There exists a triangle $a_1a_2a_3$ in $H$ and three paths $P_1, P_2, P_3$, where $P_i$ is from $a_i$ to $x_i$, such that
\begin{itemize}
\item $V(H) = (V(P_1) \cup V(P_2) \cup V(P_3)) \setminus \{x_1, x_2, x_3\} $; 
\item the sets $V(P_1)$, $V(P_2)$ and $V(P_3)$ are pairwise disjoint; and
\item for distinct $i,j \in \{1,2,3\}$, there are no edges between $V(P_i)$ and $V(P_j)$, except $a_ia_j$ and possibly $x_ix_j$.
\end{itemize}
\end{enumerate}
\end{lemma}

We are now ready to prove Theorem~\ref{thm:domsep}.
\begin{proof}
We may assume that $t \geq 2$.
Let $G \in \mathcal{M}_t$ and let $w$ be a weight function on $G$.
By working with the normalized function of $w$, we may assume that $w$ is normal.
Let $c$ be as in Lemma~\ref{lem:bigatom}.
Let $d=ct^9\log^{c}t+100$. 
We will show that there is a set $Y \subseteq G$ with $|Y|<d$ such that
$N[Y]$ is a $(w,\frac{1}{2})$-balanced separator in $G$. Suppose no such $Y$
exists.

By the proof of Lemma~5.3 of \cite{QPTAS}, there is a path $P$ in $G$ such that $N[P]$ is a $w$-balanced separator in $G$. Let $P=p_1 \dd \dots \dd p_k$, and assume that $P$ was chosen with $k$ minimum. It follows that there exists a component $B$ of
$G \setminus N[P \setminus \{p_k\}]$ such that $w(B) > \frac{1}{2}$. Let
$N=N(B)$. Then $N \subseteq N(P \setminus \{p_k\})$.

\sta{There is no $Y \subseteq G$ with  $|Y|<d$ such that $N \cup N[p_k]  \subseteq N[Y]$. \label{goodY}}

Suppose such $Y$ exists. We will show that $N[Y]$ is a $w$-balanced separator in $G$. We may assume that there is a component  $D$ of  $G \setminus N[Y]$ with
$w(D)>\frac{1}{2}$. Since $w(B)>\frac{1}{2}$, we deduce that
$D \cap B \neq \emptyset$.  Since $N \subseteq N[Y]$, it follows that
$D \subseteq B$, and so $D \cap N[P] \subseteq N[p_k]$. Since
$N[p_k] \subseteq N[Y]$, we deduce that $D$ is contained in a component of $G \setminus N[P]$, and therefore $w(D) < \frac{1}{2}$,
a contradiction. This proves that  $N[Y]$ is a $w$-balanced separator in $G$,
contrary to our assumption, and \eqref{goodY} follows.
\\
\\
For every vertex $n \in N$ let $l(n)$ be the minimum $i \in \{1, \dots, k-1\}$ such that
$n$ is adjacent to $p_i$ and let $r(n)$ be the maximum
$i \in \{1, \dots, k-1\}$ such that
$n$ is adjacent to $p_i$. Let $I(n)=l(n) \dd P \dd r(n)$.
Let $H$ be the graph with vertex set $N$ and such that $n_1n_2 \in E(H)$
if and only if $I(n_1) \cap I(n_2) \neq \emptyset$.
Let $N_0$ be a stable set of size $\alpha(H)$ in $H$. Write $N_0=\{k_1, \dots, k_m\}$
where $r(k_i) < l(k_{i+1})$ for every $i \in \{1, \dots, m-1\}$. 
Observe that  $H$ is an interval graph, and therefore the complement of $H$ is perfect \cite{intervalgraphs}. It follows that there exists a partition
$K_1, \dots, K_m$ of $V(H)$ such that for every $i$, $K_i$ is a clique of $H$ and $k_i \in K_i$. Since $K_i$ is a clique of $H$, it follows from the Helly property of the line that there exists $j(i) \in \{1, \dots m\}$ such that
$p_{j(i)} \in I(n)$ for every $n \in K_i$.

\sta{Let $N' \subseteq N$ be such that the sets
$I(n_1) \cap I(n_2) \neq \emptyset$ for all $n_1,n_2 \in N'$. Then there exists $X' \subseteq P$  with $|X'| \leq 3$ and $n' \in N'$ such that
$N' \subseteq N[X' \cup \{n'\}]$.
\label{dominateKi}}

    \begin{figure}
        \centering
        \scalebox{1.5}{\input{./figures/dbs-2.tikz}}
        \caption{Visualisation for \eqref{dominateKi}}
        \label{fig:dominateKi}
    \end{figure}

It follows from the Helly property of the line that there exists
$p_j \in \bigcap_{n \in N'}I(n)$.
Let $X'=\{p_s \; : \; s \in \{j-1,j,j+1\}  \cap \{1, \dots, k-1\}\}$.
If $N' \in N[X']$,
then \eqref{dominateKi} holds, setting $n'$ to be an arbitrary element of $N'$.
Thus, we may assume that there exists $n' \in N'$ such that $n'$ is anticomplete to $X'$. It is now enough to show that every $n \in N' \setminus (N[X']\cup \set{n'})$
is adjacent to $n'$. Suppose not, and let $n \in N' \setminus N[X' \cup \{n'\}]$. Since $p_j \in I(n) \cap I(n')$ it follows that both
$n'$ and $n$ have neighbors in $p_1 \dd P \dd p_{j-2}$ and in
$p_{j+2} \dd P \dd p_{k-1}$. It follows that there is a path $P_1$
from $n'$ to $n$ with interior in $p_1 \dd P \dd p_{j-2}$
and  a path $P_2$
from $n'$ to $n$ with interior in $p_{j+2} \dd P \dd p_{k-1}$.
Since $n',n \in N$, there is a path from $n'$ to $n$ with
interior in $B$. But now $P_1 \cup P_2 \cup P_3$ is a theta with ends
$n,n'$, a contradiction. (See \cref{fig:dominateKi}) This proves~\eqref{dominateKi}.
\\
\\
In the remainder of the proof, we will consider, for each $i \in \{1, \dots, m\}$, and
``interval'' of $N_0$ starting at $i$, but we will need to choose its end in a particular way. We will explain this next. Let $i \in \{2, \dots, m\}$
and let $J(i)=p_{l(k_i)-1} \dd P \dd p_{r(k_i)+1}$. Let $P_L(i)$ be the component of $P \setminus J(i)^*$ containing $p_1$, and let $P_R(i)$ be the component of
$P \setminus J(i)^*$ containing $p_k$.  

The following is immediate from \eqref{dominateKi}:
\\
\\
\sta{There exists $k_i' \in K_i$ such that every $k \in K_i$ has a neighbor in $J(i) \cup \{k_i'\}$. \label{KiJ(i)}}
\\
\\
\sta{There exists
$Y_i' \subseteq P \cup N$ with $|Y_i'| \leq 18$
with
$p_{l(k_i)-1}, p_{l(k_i)}, p_{l(k_i)+1}, p_{r(k_i)-1}, p_{r(k_i)}, p_{r(k_i)+1} \in  Y_i'$
such that
$N(J(i)) \cap N \subseteq N[Y_i']$. \label{dominateN(J(i))'}}

Let $Z_1=\{p_{l(k_i)-1}, p_{l(k_i)}, p_{l(k_i)+1}\}$ and
let $Z_2=\{p_{r(k_i)-1}, p_{r(k_i)}, p_{r(k_i)+1}\}$.
Let $N_1$ be the set of vertices $n \in N$ with $I(n) \subseteq I(k_i)$,
$N_2$ the set of vertices in $n \in N$ such that $p_{l(k_i)-1} \in I(n)$
and $N_3$ the set of vertices in $n \in N$ such that $p_{r(k_i)+1} \in I(n)$.
Then $N(J(i)) \cap N =N_1 \cup N_2 \cup N_3$.
By the maximality of $N_0$ it follows that the sets $I(n)$  pairwise meet
for all $n \in N_1$. Now 
\eqref{dominateKi} implies that for every $i \in \{1,2,3\}$ there exist
$n_i' \in N_i$ and
$X_i' \subseteq P$ with $|X_i'| \leq 3$ such that $N_i \subseteq N(X_i' \cup n_i')$. 
Let $$Y_i'=X_1' \cup X_2' \cup X_3' \cup \{n_1',n_2',n_3'\} \cup Z_1 \cup Z_2.$$
Now  $N(J_i) \cap N \subseteq  N[Y_i']$.
This proves \eqref{dominateN(J(i))'}.
\\
\\
Let $Y_i'$ be as in \eqref{dominateN(J(i))'}.

\sta{There exists $X_i' \subseteq P \cap N(k_i)$ with $|X_i'| \leq 4$ with the following proprety. For every path $Q=q_1 \dd \dots \dd q_s$ in 
$G \setminus N[Y_i']$ where $q_1$ has a neighbor in $J(i)$
and $q_s$ has a neighbor in $B$, we have that $Q$ meets $N[X_i' \cup \{k_i\}]$.  \label{dominateN(J(i))}}

We may assume that $Q \setminus q_s$ is anticomplete to $B$, and that 
 $Q  \setminus q_1$ is anticomplete to $J(i)$. Since $q_1 \not \in N[Y_i']$, we deduce that $s>1$. If $|N(k_i) \cap V(P)| \leq 4$, let $X_i'=N(k_i) \cap P$. If $|N(k_i) \cap V(P)| > 4$, let $J$ be the set of consisitng of the
two minimum values of $j$, and the two maximum values of $j$ such that
$p_j \in N(k_i) \cap J(i)$, and let $X_i'=\{p_j \text{ : }j \in J\}$.
Assume that $Q \cap N[X_i' \cup \{k_i\}]=\emptyset$.
It follows that $k_i$ is anticomplete to $Q$.
Let $R$ be a path from $k_i$ to $q_s$ with interior in $B$.

Suppose first that $q_1$ has a neighbor $v$ in $J(i) \setminus N(k_i)$. Since
$q_1 \not \in N[Y_i']$, it follows that
there exists a subpath $P'=p_j \dd P \dd p_l$ of $J(i)$ with $j<l$ such that
$p_j,p_l$ are adjacent to $k_i$, $k_i$ is anticomplete to $\{p_{j+1}, \dots, p_{l-1}\}$ and $v \in \{p_{j+1}, \dots, p_{l-1}\}$ (see \cref{fig:dominateN(J(i))}).
    \begin{figure}[h]
        \centering
        \scalebox{1.3}{\input{./figures/dbs-5.tikz}}
        \caption{Visualisation for \eqref{dominateN(J(i))}}
        \label{fig:dominateN(J(i))}
    \end{figure}
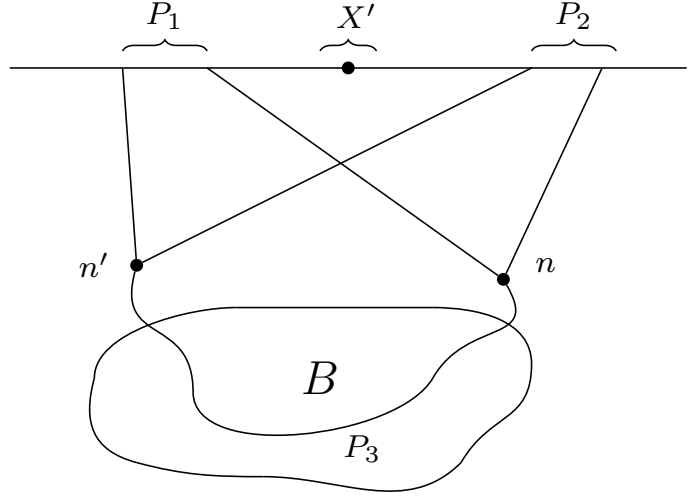
    
If $q_1$ has two non-adjacent neighbors in
$P'$, then $P' \cup Q \cup R$ contains a theta with ends $q_1,k_i$; if $v$ is the unique neighbor of $q_1$ is $P'$, then $P' \cup Q \cup R$ is a theta with ends $v,k_i$;
and if $q_1$ has exactly two neighbors $u,v$ in $P'$ and $u$ is adjacent to $v$, then $P' \cup Q \cup R$ is a pyramid with apex $k_i$ and base $uvq_1$. This proves that $N(q_1) \cap J(i) \subseteq N(k_i) \cap J(i)$.

Since $q_s$ has a neighbor in $B$ and $q_{s-1}$ is anticomplete to $B$, it follows that $q_s \in N$, and therefore  $q_s$ has a neighbor in $P$. Since $s>1$, $q_s$ is anticomplete to $J(i)$ and has a neighbor in $P \setminus J(i)$. Assume that $q_s$ has a neighbor in $P_R(i)$  (the argument for the other case is analogous). Let $l>r(k_i)$ be minimum such that $p_l$ has a neighbor in $Q$.
Let $j<r(k_i)$ be maximum such that $q_1$ is adjacent to $p_j$. Then
$j \in N(k_i) \cap J(i)$. Let $S=k_i \dd p_{r(k_i)} \dd P \dd p_l$. Since
$p_j \not \in X_i'$, it follows that $p_j$ is anticomplete to $S \setminus k_i$.Now if $p_l$ has two non-adjacent neighbors in $Q \cup R$, then
$Q \cup R \cup S$ contains a theta with ends $p_l, k_i$;
if $p_l$ has a unique neighbor $v$ in $Q \cup R$, then $Q \cup R \cup S$ is a theta with ends $v, k_i$; and if $p_l$ has exactly two neighbors $u,v \in Q \cup R$ and $u$ is adjacent to $v$, then $Q \cup R \cup S$ is a pyramid with apex
$k_i$ and base $p_luv$, in all cases a contradiction.
This proves \eqref{dominateN(J(i))}.
\\
\\
Let $X_i'$ be as in \eqref{dominateN(J(i))} and let $Y_i=Y_i' \cup X_i' \cup \{k_i, k_i'\}$.
Then $k_i, k_i', p_{l(k_i)-1}, p_{l(k_i)}, p_{r(k_i)}, p_{r(k_i)+1} \in  Y_i$.
Let $Z_i=(N[Y_i] \cap N(P))  \cup J(i)^*$.
Let $U_i=(V(G) \setminus Z_i)   \cup
\{k_i,k_1, p_{l(k_i)-1}, p_{r(k_i)+1}\}$.
Then $B \subseteq U_i$. 
Let $D_i$ be the component of $U_i$ containing
$B$.  Then $k_1, p_k \in D_i$, and therefore $P \setminus J(i)^* \subseteq D_i$.
Also, $k_i \in D_i$.
Write $(z_1,z_2,z_3)=(p_{l(k_i)-1}, k_i,p_{r(k_i)+1})$.
Let $G_i'=G[D_i]$.

From now on, we assume that $i>1$.
\\
\\
\sta{One of the following holds.
\begin{enumerate}[(i)]
\item  There is a component $D_i'$ of $G_i' \setminus N[z_2]$ 
with $z_1,z_3 \in D_i'$, and  $n_i \in N(z_2) \cap N(D_i')$ 
such that the set $\{z_1,z_2,z_3\}$ is constricted in the graph
$G_i=G[D_i' \cup \{n_i,z_2\}]$ \label{claw}, or
\item There is a path $T_i$ from $z_1$ to $z_3$ in $G_i'$ such that
either $z_2$ has a unique neighbor in $T_i$, or $z_i$ has at least
two non-adjacent neighbors in $T_i$ \label{path}, or
\item no component of $G_i' \setminus N[z_2]$ contains both $z_1$ and $z_3$ \label{path2}. 
\end{enumerate}
\label{constricted}}

We may assume that there is  a component 
$D_i'$  of $G_i' \setminus N[z_2]$ such that $z_1,z_3 \in D_i'$. Then
there is a neighbor $n_i$ of $z_2$ such that $n_i$ has a neighbor in
$D_i'$. Let $G_i=G[D_i' \cup \{n_i,z_2\}]$. We may also assume that
no path $T_i$ as in outcome \ref{path} exists.

Suppose that  that $\{z_1,z_2,z_3\}$ is not
constricted in $G_i$.
Then there is a tree $T$ containing $z_1,z_2,z_3$.
Since $z_1$, $z_2$, $z_3$
 have degree one in $G_i$, it follows that 
$T$ is a subdivision of $K_{1,3}$ and $z_1,z_2,z_3$ are the leaves of $T$.
Let $t$ be the unique vertex of $T$ of
degree three, and for $j \in \{1,2,3\}$ let $P_j$ be the path of $T$ from
$t$ to $z_j$. Since $P_1 \cup P_3$ does not satisfy outcome \ref{path},  it
follows that   $t$ is non-adjacent to $z_2$.
By \eqref{dominateN(J(i))}, $D_i \setminus \{z_1,z_2,z_3\}$ is
anticomplete to $J(i)^*$ in $G$, and in particular
$T \setminus \{z_1,z_2, z_3\}$ is
anticomplete to $J(i)^*$ in $G$.  Now, if $l(z_2)=r(z_2)$,
we get that $T \cup J(i)^*$ is a theta with ends $t,p_{l(z_2)}$;
if $r(z_2)=l(z_2)+1$, we get that $T \cup J(i)^*$ is a pyramid with base $z_2p_{l(z_2)}p_{r(z_2)}$ and apex $t$; and if $r(z_2) \geq l(z_2)+2$, we get
that $T \cup \{p_{l(z_2)}, p_{r(z_2)}\}$    is  theta with ends $t,z_2$; in all cases, a contradiction. This proves \eqref{constricted}.
\\
\\
Our next goal is to define a special connected set $A_i$.
\\
\\
\sta{Suppose that outcome \ref{path} or outcome \ref{path2}  of  \eqref{constricted} holds.
Then there exists $\Delta_i\subseteq G \setminus N[P]$ with $|\Delta_i| \leq 2$ 
and a component $A_i$ of $G_i' \setminus N[\{k_i\} \cup \Delta_i]$  such that 
\begin{enumerate}
    \item $w(A_i) > \frac{1}{2}$, and
    \item at  least one of the sets $P_L(i) \cap N_G[A_i]$ and $P_R(i) \cap N_G[A_i]$ is empty.
\end{enumerate}
\label{wheel}}

Suppose first that outcome \ref{path2} of \eqref{constricted} holds, and so
no component of $G_i' \setminus N[z_2]$ contains both
$z_1$ and $z_3$. Let $\Delta_i=\emptyset$.
Since $(N[Y_i] \cap N(P)) \cup N[\{z_2\}]$ is not a balanced separator in
$G$, there is a component $A_i$ of  $G_i' \setminus N[z_2]$ with
$w(A_i) > \frac{1}{2}$. Then $|A_i \cap \{z_1,z_3\}| \leq 1$.
Since $(N[Y_i] \cap N(P)) \cup N[\{z_2\}]$ is disjoint from $P_L(i) \cup P_R(i)$, it follows that one of the sets $N[A_i] \cap P_L(i)$ and
$N[A_i] \cap P_R(i)$ is empty, as required.

Thus, we may assume that outcome \ref{path} of \eqref{constricted} holds.
Let $T_i=t_1 \dd \dots \dd t_m$ be a path as in outcome
\ref{path} of \eqref{constricted} where $t_1=z_1$ and $t_m=z_3$. Let $q$ be minimum and $r$ be maximum such that $z_2$ is adjacent to $t_q,t_r$. Then $r \neq q+1$. Let
$\Delta_i=\{t_q,t_r\}$.
Since $k_i \in Y_i$, it follows that $\Delta_i \subseteq G \setminus N[P]$.
Since 
$(N[Y_i] \cap N(P))  \cup N[\{k_i,t_q,t_r\}]$ is not a $w$-balanced separator in $G$, it follows that there is a component 
$A_i$ of
$G_i' \setminus N[\{k_i, t_q,t_r\}]$ with
$w(A_i)>\frac{1}{2}$.

Suppose $P_L(i) \cap N[A_i] \neq \emptyset$ and
$P_R(i) \cap N[A_i] \neq \emptyset$.
Since $k_i \in Y_i$, it follows that $t_q,t_r \not \in N(P)$, and
consequently $N[\{k_i,t_q,t_r\}] \cap (P_L(i) \cup P_R(i))=\emptyset$.
We deduce that 
$N_G(A_i) \cap (P_L(i) \cup P_R(i))=\emptyset$; therefore 
$A_i \cap P_L(i) \neq \emptyset$ and $A_i \cap P_R(i) \neq \emptyset$,
and so $P_L(i) \cup P_R(i) \subseteq A_i$.
Consequently, there is a path
$R$ from $z_1$ to $z_3$ with
$R^* \subseteq A_i$.
Recall that  $R^* \cap N[\{k_i,  t_q,t_r \}]=\emptyset$,
and $R^* \cap N[Y_i] \cap N(P)=\emptyset$.
Let $R'$ be a minimal subpath of $R$ from a vertex $r$ with a
neighbor in $t_1 \dd T_i \dd t_{q-1}$ to a vertex $r'$ with a neighbor in
$t_{r+1}  \dd T_i \dd t_m$.
Since $z_1$ has a unique neighbor in $G_i'$, it follows that
$r$ is non-adjacent to $z_1$, and similarly $r'$ is non-adjacent to $z_3$.
Since $J(i) \cup T_i \cup \{k_i\}$ is not a theta or a pyramid,
we deduce that either $p_r(k_i)>p_l(k_i)+1$, or $r>q+1$.
Now exactly of the following holds.
\begin{enumerate}[(i)]
\item There is a path $S$ from $k_i$ to $r'$ with
$S^* \subseteq t_r \dd T_i \dd t_m$ and such that $S$ is anticomplete to
$t_1 \dd T_i \dd t_q$.
\item $r=q$, and $t_{q+1}$ is the unique neighbor of $r'$ in $t_{q+1} \dd T_i \dd t_m$.
\end{enumerate}
Also, switching the roles of $z_1$ and $z_3$, exactly one of the following holds:
\begin{enumerate}[(i)]
\item There is a path $S'$ from $k_i$ to $r$ with
$S^* \subseteq t_1 \dd T_i \dd t_q$ and such that $S'$ is anticomplete to
$t_r \dd T_i \dd t_m$.
\item $r=q$, and $t_{q-1}$ is the unique neighbor of $r$ in $t_1 \dd T_i \dd t_{q-1}$.
\end{enumerate}
We claim that  at least one of the following holds:
\begin{enumerate}[(i)]
\item There is a path $S$ from $k_i$ to $r'$ with
$S^* \subseteq t_r \dd T_i \dd t_m$ and such that $S$ is anticomplete to
$t_1 \dd T_i \dd t_q$.
\item There is a path $S'$ from $k_i$ to $r$ with
$S^* \subseteq t_1 \dd T_i \dd t_q$ and such that $S'$ is anticomplete to
$t_r \dd T_i \dd t_m$.
\end{enumerate}
Suppose not. Then $q=r$, $t_{q-1}$ is the unique neighbor of $r$ in $t_1 \dd T_i \dd t_{q-1}$, and  $t_{q+1}$ is the unique neighbor of $r'$ in $t_{q+1} \dd T_i \dd t_m$. From the minimality of $R'$ we have that $R' \setminus \{r,r'\}$ is
anticomplete to $T_i$. By  \eqref{dominateN(J(i))}, $G_i' \setminus \{z_1,z_2,z_3\}$ is anticomplete to $J(i)^*$. But now
$J(i) \cup T_i \cup R'$ is a theta with ends $t_{q-1},t_{q+1}$, a contradiction.
This proves the claim.

Exchanging the roles of $z_1,z_3$ if necessary, we may assume that
\begin{itemize}
\item There is a path $S$ from $k_i$ to $r'$ with
$S^* \subseteq t_r \dd T_i \dd t_m$ and such that $S$ is anticomplete to
$t_1 \dd T_i \dd t_q$.
\end{itemize}

Let $1 \leq x \leq y \leq q-1$  such that $x$ is minimum and $y$ is maximum
with $t_x,t_y$ adjacent to $r$.
Recall that $r$ is non-adjacent to $p_l(k_i)$, and $r$ is non-adjacent to
$t_q$. Now if $x=y$, then
$t_1 \dd T_i \dd t_q \cup S \cup R' \cup \{p_l(k_i)\}$
is a theta with ends $t_x, k_i$; if $y=x+1$,
then 
$t_1 \dd T_i \dd t_q \cup S \cup R' \cup \{p_l(k_i)\}$ is a pyramid with apex $k_i$ and
base $rt_xt_y$; and if
$y>x+1$, then
$t_1 \dd T_i \dd t_x \cup t_y \dd T_i \dd t_q \cup S \cup R' \cup  \{p_l(k_i)\}$
is a theta with ends $r,k_i$; in all cases a contradiction.
This proves \eqref{wheel}.
\\
\\
We have defined a set $A_i$ in the case when outcome 
\ref{path} or outcome \ref{path2} of \eqref{constricted} holds.
Now we define $A_i$ in the remaining case. Thus, assume that outcome
\ref{claw} of \eqref{constricted}  holds and $\{z_1,z_2,z_3\}$ is
constricted in $G_i$. By Theorem~\ref{stripdecomp}, there is a graph $H_i$ such that
$(G_i, \{p_{l(k_i)-1}, k_i,p_{l(k_i)+1}\})$ admits a
faithful extended strip decomposition $\eta_i$ with pattern
$H_i$.
\\
\\
\sta{There exists an atom $A$ of $\eta_i$,  and a   component $A_i$ of $A$ such that $w(A_i)>\frac{1}{2}$.
\label{bigatom}}

Suppose not.
Then by Lemma~\ref{lem:bigatom} applied to $G_i$,
we deduce that there exists $W_i \subseteq G_i$
with $|W_i| \leq d-100$, such that $N[W_i]$ is a $w$-balanced separator in
$G_i$. It follows from the definition of $G_i$
that $N[Y_i \cup W_i \cup \{p_k,k_1, p_{l(k_i)-1}, p_{r(k_i)+1},n_i\}]$ is a $w$-balanced separator in $G$,
which is a contradiction since $|Y_i| <80$. This proves~\eqref{bigatom}.
\\
\\
If outcome \ref{claw} of \eqref{constricted} holds, let $A_i$ be as in \eqref{bigatom}.
This completes the definition of $A_i$.

\sta{If  $P_L(i) \cap N_G[A_i] \neq \emptyset$, then
$|P_R(i) \cap N_G[A_i]| \leq 2$. \label{notbothsides}}

If outcome \ref{path} or outcome \ref{path2} of \eqref{constricted} holds, this follows immediately from
\eqref{wheel}. Thus we may assume that outcome
\ref{claw} of \eqref{constricted} holds and $\{z_1,z_2,z_3\}$ is constricted in $G_i$.
Suppose that both $P_L(i) \cap N_G[A_i] \neq \emptyset$ and
$|P_R(i) \cap N_G[A_i]|>2$. It follows from the definition of $G_i$
that $P_L(i) \cap N_{G_i}[A_i] \neq \emptyset$ and
$|P_R(i) \cap N_{G_i}[A_i]|>2$.

Let $Q_1=z_1 \dd P_L(i)$. Since $N[P]$ is a
$w$-balanced separator, we have that $N(p_k) \cap B \neq \emptyset$;
let $Q$ be a path from $z_2$ to $p_k$ with $Q^* \subseteq B$.
Then $F=z_3 \dd P_R(i) \dd p_k \dd Q \dd z_2$ is a path. Let $Q_3$ be the minimal subpath of $F$
from $z_3$ to a vertex of $N_{G_i}[A_i]$, and let
$Q_2$ be the minimal subpath of $F$ from $z_2$ to a
vertex of $N_{G_i}[A_i]$. Since $|P_R(i) \cap N_{G_i}[A_i]|>2$,
it follows that $Q_2$ is anticomplete to $Q_3$.
But now  $Q_1,Q_2,Q_3$ are pairwise disjoint and anticomplete to each other;
$z_i$ is an end of $Q_i$, and  $Q_i \cap N_{G_i}[A_i] \neq \emptyset$
for every $i \in \{1,2,3\}$, contrary to Lemma~\ref{threepaths}.
This proves~\eqref{notbothsides}.
\\
\\
If outcome \ref{claw} of \eqref{constricted} holds let $\delta_i$ be the boundary of $A_i$ in $\eta_i$,  let $\Delta_i'$ be a core for $\delta_i$
with $|\Delta_i'| \leq 3$ (such $\Delta_i'$ exists by Lemma~\ref{atomboundary}), and let $\Delta_i=\Delta_i' \cup \{n_i\}$, where $n_i$ is as in \eqref{constricted}\ref{claw}.
If outcome \ref{path} or outcome \ref{path2} of \eqref{constricted} holds, let
$\Delta_i$ be as in \eqref{wheel} and let $\delta_i=N[\Delta_i]$.
In both cases let  $\gamma_i=N_G[A_i] \cap P$.
\\
\\
\sta{Let $Z \subseteq V(G)$ with
$Y_i \cup \{p_k,k_i, p_{l(k_i)-1}, p_{r(k_i)+1}\} \subseteq Z$ 
and such that $\delta_i \subseteq N[Z]$. Let
$D \subseteq G \setminus N[Z]$ be connected with $w(D)>\frac{1}{2}$.
Then
$D \subseteq A_i$, $N[D] \cap J(i)=\emptyset$,  and
there exists $v \in D \cap N(B)$.
\label{meet}}

Since $w(B)>\frac{1}{2}$, it follows that $B \cap D \neq\emptyset$.
Similarly, since $w(A_i)>\frac{1}{2}$, $A_i \cap D \neq \emptyset$.
Since $Y_i \subseteq Z$, it follows from \eqref{dominateN(J(i))'} and \eqref{dominateN(J(i))} that
$J(i) \cap N[D] =\emptyset$. 
Since $Y_i \cup \{p_k,k_i, p_{l(k_1)-1}, p_{r(k_i)+1}\} \subseteq Z$  and since
$\delta_i \subseteq N[Z]$, it follows that
$N_{G \setminus J(i)}(A_i) \subseteq N[Z]$.
We deduce that  $D \subseteq A_i$. 
Since $p_k \in Z$,
and since $N[P]$ is a balanced separator in $G$, it follows that
$D \setminus B \neq \emptyset$. Since $D$ is connected, there exists
$v \in D \setminus B$ with a neighbor in $B$. 
This proves~\eqref{meet}.
\\
\\
 \sta{$|P_R(2) \cap N_G[A_2]|>2$.
 \label{startright}}

Suppose that $P_R(2) \cap N_G[A(2)] \leq 2$. Let
$$Z = Y_1 \cup Y_2 \cup \{p_k\} \cup \Delta_i \cup (\gamma_2 \cap P_R(2)).$$
We claim that $N[Z]$ is a balanced separator in $G$.
Suppose not, and let $D$ be a component of $G \setminus N[Z]$ with
$w(D) > \frac{1}{2}$.
By \eqref{meet}  $D \subseteq A_2$ , $N[D] \cap J(2)=\emptyset$ and
there exists
$v \in D \setminus B$ with a neighbor in $B$. Then $v \in N \cap A_2$. Let $v' \in P$ be a neighbor of $v$, choosing $v' \in P_R(2)$ if possible.
Since 
$N[D] \cap J(2)=\emptyset$ and since $v \in A_2$,
it follows that  $v' \in N[A_2] \cap (P \setminus J(2))$.
Since
$ Y_1 \subseteq Z$, it follows from
\eqref{dominateN(J(i))'}  that
$v$ is anticomplete to  $J(1)$. Since $k_1' \in Y_1$ and $k_2' \in Y_2$,
\eqref{KiJ(i)} implies that 
$v \not \in K_1 \cup K_2$. Since $v \in N$, it follows
from the choice of $v'$ that $v' \in \gamma_2 \cap P_R(2)$. But now  $v' \in Z$ and 
so  $v \in N[Z]$, contrary to the fact that
$v \in D$. This proves that $N[Z]$ is a balanced separator in
$G$, contrary to the fact that $|Z|<d$, and \eqref{startright}
follows.
\\
\\
\sta{$P_L(m) \cap N_G[A_m] \neq \emptyset$.
\label{endleft}}

Suppose that $P_L(m) \cap N_G[A_m] = \emptyset$. Let
$$Z = Y_m \cup \{p_k\} \cup \Delta_i.$$
We claim that $N[Z]$ is a balanced separator in $G$.
Suppose not, and let $D$ be a component of $G \setminus N[Z]$ with
$w(D) > \frac{1}{2}$. By \eqref{meet} 
$D \subseteq A_m$, $N[D] \cap J(m)=\emptyset$, and there exists
$v \in D \setminus B$ with a neighbor in $B$. Then $v \in N \cap A_m$;
let $v' \in P$ be a neighbor of $v$, choosing $v' \in P_L(m)$ if possible.
Since $N[D] \cap J(m)=\emptyset$ and since $v \in A_m$, we deduce that  $v' \in N[A_m] \cap (P \setminus J(m))$.
Since $k_m' \in Y_m$,   it follows from \eqref{KiJ(i)} that $v \not \in K_m$. By
the choice of $v'$, we deduce that  $v' \in P_L(m)$, a contradiction.
This proves that $N[Z]$ is a balanced separator in
$G$, contrary to the fact that $|Z|<d$, and \eqref{endleft}
follows.
\\
\\
By \eqref{startright} $|P_R(2) \cap N_G[A(2)]|>2$, and therefore
by \eqref{notbothsides} $P_L(2) \cap N_G[A(2)]=\emptyset$.
In view of this, let $i$ be maximum such that
$P_L(i) \cap N_G[A_i]= \emptyset$.
By \eqref{endleft}, $i<m$. By the maximality of $i$,
$P_L(i+1)  \cap N_G[A_{i+1}] \neq  \emptyset$, and therefore  by
\eqref{notbothsides} $|P_R(i+1)  \cap N_G[A_{i+1}]| \leq 2$.
Let
$$Z=Y_1 \cup Y_i \cup Y_{i+1} \cup \Delta_i \cup \Delta_{i+1} \cup
(\gamma_{i+1} \cap P_R(i+1)) \cup   \{p_k\}.$$ 
Then $|Z| < d$. To complete the proof,  we obtain a contradiction by showing that 
$N[Z]$ is a $w$-balanced separator in $G$.

Suppose not, and let $D$ be a component of
$G \setminus N[Z]$ with
$w(D) > \frac{1}{2}$. By \eqref{meet},
$D \subseteq A_i \cap A_{i+1}$, $N[D] \cap (J(i) \cup J(i+1))=\emptyset$,  and there exists
$v \in D \setminus B$ with a neighbor in $B$. Then $v \in N \cap A_i \cap A_{i+1}$; let $v' \in P$ be a neighbor of $v$.
Then $v' \in N_G[A_i] \cap N_G[A_{i+1}] \cap (P \setminus (J(i) \cup J(i+1)))$, and therefore $v' \in  P_L(i+1) \cup P_R(i+1)$.
Since $P_L(i) \cap N[A_i] = \emptyset$ and $P_R(i+1) \cap N[A_{i+1}] \subseteq Z$, it follows that $v' \in Z$.
But now  $v \in N[Z]$, contrary to the fact that $v \in D$.

\end{proof}

\section{Constricted sets in $K_{t,t}$-free graphs} \label{sec: constricted in ktt free}
\MariaIsEditing{\textcolor{red}{DO NOT MODIFY! Maria has the token!}}
\input{esd_section}

\section{Cooperative sets} \label{sec: cooperative matroid}
\MariaIsEditing{\textcolor{red}{DO NOT MODIFY! Maria has the token!}}

In this section, we introduce the notion of \textit{cooperative sets}, which will be central to the rest of the proof.  Cooperative sets are used to prove \cref{thm: small pairwise separator}:
instead of trying to separate two vertices directly, we will construct a ``cooperative set'' starting with one of them and use its properties to separate it from the other vertex. In this section, we develop the necessary properties of cooperative sets,  culminating with 
\cref{corollary: matroid gives partition of neighborhood}. The main tools are \cref{corollary theta free degenerate}
(Theorem 4.4 from \cite{MatijaPolyDegen}) and \cref{thm:matroid_interesection}
(The Matroid Intersection Theorem \cite{edmonds1970}). The use of cooperative sets is explained in \cref{sec: cooperative set evolving degenerate partition}.

We start with two results from the literature.

\begin{lemma}[Lemma 2 from \cite{LOZIN2022103517}]\label{Lemma: ktt from pairwise adjacent}
\label{lemma scott's trick}
    For all positive integers $a$ and $b$, there is a positive integer $C = C(a, b)$ such that if a graph $G$ contains a collection of $C$ pairwise disjoint subsets of vertices, each of size at most $a$ and with at least one edge between every two of them, then $G$ contains a $K_{b,b}$-subgraph
\end{lemma}

\begin{corollary}[immediate corollary of Theorem 4.4 from \cite{MatijaPolyDegen} ]
\label{corollary theta free degenerate}
    There exists a polynomial $q$, such that the average degree of every theta-free graph $G$ is at most $q(\omega(G))$.
\end{corollary}

Let $G$ be a graph. A \textit{matching} $M\subseteq E(G)$ is a set of disjoint edges. 
We denote by $V(M)$ the set of all endpoints of the edges of $M$.
We say that a vertex $x$ is \textit{matched} by $M$ (or that $M$ {\em matches} $x$)  if there exists an edge in $M$ such that $v$ is one of its endpoints.
Let $X,Y \subseteq V(G)$.  A matching {\em from $X$ to $Y$} is a matching each of whose edges has one end in $X$ and the other end in $Y$.
We say that $X$ \textit{matches into} $Y$ if there exists a matching from $X$ to $Y$ and that matches every vertex in $X$. A matching is said to be induced if $E(G[V(M)]) = M$.
We denote $V(M)$ the set of all vertices matched by the matching $M$.

In the following, by a polynomial in two variables $x$ and $y$, we mean a finite sum of terms of the form $a_{ij}x^iy^j$, where $a_{ij}$ are real non-zero coefficients,  and $i,j$ are non-negative integers. We show:
\begin{lemma}\label{Lemma: induced matching from non induced}
    There exists a polynomial $p(x,y)$ 
    for which the following holds. Let $k\in \nat$, $G$ be a theta-free graph, and $M=\set{(x_i,y_i)}_{i=1}^{p(\omega(G),k)}$ a matching (not necessarily induced) in $G$ then there exist $M'\subseteq M$ such that $|M'|\geq k$ and $M'$ is an induced matching.

\end{lemma}

\begin{proof}
     Let $C(a,b)$ be defined as in Lemma~\ref{Lemma: ktt from pairwise adjacent}, let $q$ be the polynomial from \cref{corollary theta free degenerate} and let $p(x,k) = k^{C(2,3)}(q(x)+1)^2$ . 
     Let $$X = \set{x_i| i \in 1,\dots,k^{C(2,3)}(q(x)+1)^2}$$ and $$Y = \set{y_i| i \in 1,\dots,k^{C(2,3)}(q(x)+1)^2}.$$
     By \cref{corollary theta free degenerate}, $G$ is $q(\omega(G))$ degenerate and therefore $q(\omega(G))+1$ colorable. It follows that we can find a subset $I\subseteq \set{1,\dots,p(\omega(G))}$ such that such that $|I| = k^{C(2,3)}$ and $X'=\set{x_i}_{i\in I}$ and $Y'=\set{y_i}_{i\in I}$ are stable sets.
     Let $H$ be the graph with the vertex set $\set{(x_i,y_i) | i\in I}$ and where $(a,b)$ is adjacent to $(b,d)$ if $\set{a,b}$ and $\set{c,d}$ are not anticomplete. Since, by \cref{thm:ramsey}, $|I|\geq R(k,C(2,3))$ , $H$ contains either a clique of size $C(2,3)$ or a stable set of size $k$.  If $H$ contains a stable set of size $k$, we are done since this corresponds to an induced matching of size $k$.
    Therefore, we can assume that $H$ contains a clique of size $C(2,3)$. This implies that we have a set of $C(2,3)$ sets of two vertices (the edges) with an edge between every two of them. Therefore, by \cref{lemma scott's trick}, $G[Y'\cup X']$ contains $K_{3,3}$ as a (non necessarily induced) subgraph. Finally, since both $X'$ and $Y'$ are independent sets, this $K_{3,3}$ is actually induced, which is a contradiction as $G$ is theta-free.
\end{proof}

\MariaIsEditing{\textcolor{red}{DO NOT MODIFY! Maria has the token!}}
Let $G$ be a graph, and $X\subseteq V(G)$ be connected. We define the \textit{boundary} of $X$ in $G$, which we denote by $\delta^G(X) = \delta^G_1(X)$, to be the set of vertices in $X$ having at least a neighbor in $G\sm X$. Let $\delta^G_i(X) = \delta^G(X\sm \bigcup_{k<i} \delta^G_k(X))$.

We say that $X$ is \textit{cooperative} in $G$ if the  following three  conditions hold \begin{itemize}
    \item  every vertex in $\delta^G_1(X)$ has a neighbor in $X\sm \delta^G_1(X)$
    \item  every vertex in $\delta^G_2(X)$ has a neighbor in $X\sm (\delta^G_1(X)\cup \delta^G_2(X))$ 
    \item $X\sm (\delta^G_1(X)\cup \delta^G_2(X))$ is connected 
\end{itemize}
See \cref{fig:cooperative set}.

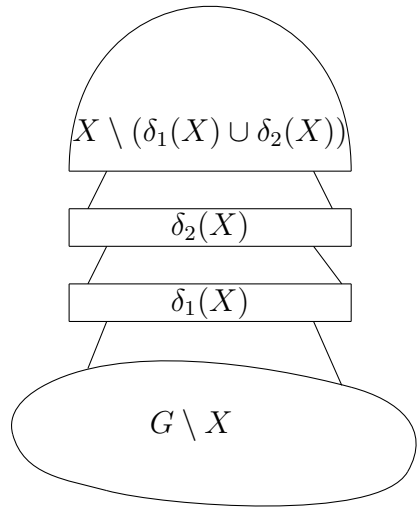
\begin{figure}[h]
        \centering
        \scalebox{1}{\input{./figures/cooperative_set.tikz}}
        \caption{Visualization of a cooperative set}
        \label{fig:cooperative set}
\end{figure}
When the graph $G$ is clear from the context, we may omit it from the notation.

\begin{lemma}\label{Lemma N(cooperative) is cooperative}
    Let $G$ be a graph and $X$ be a cooperative set in $G$. Let $D_1,\dots,D_k$ be the connected components of $G\sm N[X]$. Then for all $1\leq i\leq k$, $X\cup N_G(D_i)$ is a cooperative set in $X\cup N[D_i]$.
\end{lemma}
\begin{proof}
    Let us fix an $i$ arbitrarily and let $G'=X\cup N_G[D_i]$ and $X'= X \cup N_G(D_i) = N_{G'}[X]$. We have that $\delta_1^{G'}(X')=N_G(D_i)=N_{G'}(D_i)$ and $\delta_2^{G'}(X')=N_{G'}(N_{G'}(D_i)) \cap X'$. See \cref{fig: N(cooperative) set}.
    \begin{figure}[H]
        \centering
        \scalebox{1}{\input{./figures/n_cooperative.tikz}}
        \caption{Visualization for \cref{Lemma N(cooperative) is cooperative}}
        \label{fig: N(cooperative) set}
    \end{figure}

    Let us now verify that the three conditions to be a cooperative set hold.
    
    \sta{The first condition holds \label{claim: first condition}}
    
    We have that \[N_{G'}(X'\sm \delta_1^{G'}(X')) = N_{G'}(X) = N_{G'}(D_i)=\delta_1^{G'}(X').\]
    This proves \eqref{claim: first condition}.

    \sta{The second condition holds \label{claim: second condition}}

    Moreover, we have that \begin{align*}
        N_{G'}(X'\sm (\delta^{G'}_1(X')\cup \delta^{G'}_2(X')))&= N_{G'}\left((X\cup N_{G'}(D_i))\sm \left( N_{G'}(D_i)\cup \left(N_{G'}(N_{G'}(D_i)) \cap X'\right)\right)\right)\\
        &= N_{G'}(X\sm N_{G'}(N_{G'}(D_i)))\\
        &= \delta_2^{G'}(X')
    \end{align*}
    This proves \eqref{claim: second condition}.

    \sta{The third condition holds \label{claim: third condition}}
    Notice that $\delta^{G'}_2(X') \subseteq \delta_1^{G}(X)$. Therefore, $X'\sm (\delta^{G'}_1(X')\cup \delta^{G'}_2(X')) = (X\sm \delta_1^G(X)) \cup Y$ for some $Y\subseteq \delta_1^{G}(X)$. Since $X$ is cooperative in $G$, $X\sm \delta_1^G(X)$ is connected and every vertex in $\delta_1^{G}(X)$ has a neighbor in $\delta_2^G(X)$ and thus, every vertex in $Y$ has a neighbor in $X\sm \delta_1^G(X)$. This proves \eqref{claim: third condition}.
\end{proof}

\begin{lemma}\label{Lemma cooperative deletion}
    Let $G$ be a graph, $X$ be a cooperative set in $G$, and $C\subseteq (G\sm X) \cup \delta_1^G(X)$. Let $D$ be a connected component of $G\sm (X\cup C)$ which is not anticomplete to $X\sm C$. Then $X'=\left((X\sm \delta_1^G(X))\cup N(D)\right)\cap (X\sm C)$ is cooperative in $G'= (N[D]\sm C) \cup X'$.
\end{lemma}
\begin{proof}
    We have that $\delta_1^{G'}(X') = N_{G'}(D)$. Since $X'\sm \delta_1^{G'}(X') = X\sm \delta_1^G(X)$ the first condition of being cooperative holds. We also have that $\delta_2^{G'}(X') \subseteq \delta_2^G(X)$ and that $X'\sm (\delta^{G'}_1(X')\cup \delta^{G'}_2(X')) \supseteq X\sm (\delta^G_1(X)\cup \delta^G_2(X)) $ so the second condition holds. Finally, we have that $X'\sm (\delta^{G'}_1(X')\cup \delta^{G'}_2(X')) = (X\sm (\delta^G_1(X))\cup \delta^G_2(X)) \cup Y$ for some $Y\subseteq \delta_2^G(X)$. Since $ X\sm (\delta^G_1(X)\cup \delta^G_2(X))$ is connected and every vertex in $Y$ has a neighbor in this set, $X'\sm (\delta^{G'}_1(X')\cup \delta^{G'}_2(X'))$ is connected, and
    the third condition of being cooperative holds.
\end{proof}

In the remainder of this section, we will assume familiarity with basic matroid theory (see \cite{oxleyMatroid} for an introduction to the subject).
Let $X\subseteq V(G)$ be a cooperative set of $G$ and let $b$ be a vertex anticomplete to $X$.
Let $\Mca_{X,b}^1$ on $G$ be the pair $(N(X),\Ic_1)$  where $\Ic_1$ is the set of subsets of $N(X)$ that are matched into $\delta(X)$.

Let $\Mca_{X,b}^2$ on $G$ be the pair $(N(X),\Ic_2)$ where $\Ic_2$ is the set of subsets $Y$ of $N(X)$ for which they are $|Y|$ vertex-disjoint paths (except at $b$) from $Y$ to $b$ that are internally disjoint from $N[X]$.

\begin{lemma}\label{Lemma: both are matroids}
    Both $\Mca_{X,b}^1$ and $\Mca_{X,b}^2$ on $G$ are matroids.
\end{lemma}
\begin{proof}
    For $\Mca_{X,b}^1$, consider the bipartite graph $G_1$ that is built from $G[\delta(X)\cup N(X)]$ by removing the internal edges in both $\delta(X)$ and $N(X)$. Then $\Mca_{X,b}^1$ is the transversal matroid on $G_1$ (see Theorem 1.61 of \cite{oxleyMatroid}).
    For $\Mca_{X,b}^2$, consider the digraph $G_2$ with vertex set $V(G)\sm X$ and where $(x,y)\in E(G_2)$ if $\set{x,y}\in E(G)$ and $y\notin N(X)$. Let $G_2'$ be the digraph made from $G_2$ by adding $|V(G)\sm N(X)|$ copies of $b$ along with its incident edges. Then, $\Mca_{X,b}^2$ is the gammoid on $G_2'$ made with the pairwise vertex disjoint paths from $N_{G}(X)$ to copies of $b$ in $G_2'$ (see  \cite{gammoids}).
\end{proof}

\begin{lemma}\label{lemma : small intersection of matroids}
Let $G$ be a (theta, pyramid)-free graph, $X$ be a cooperative set in $G$, and $b\in V(G)$ be anticomplete to $X$. There exists a polynomial $p$ for which the size of the largest common independent set of $\Mca_{X,b}^1$ and $\Mca_{X,b}^2$ on $G$ is strictly smaller than $p(\omega(G))$.
\end{lemma}

\begin{proof}
Since $G$ is theta-free, it follows that $G$ is $K_{3,3}$-free.
    Let $c=c(3)$  given by Theorem~\ref{Theorem: small set hit all constricted path to v} 
    and let $p(\omega(G)) = q(\omega(G),2q(\omega(G),\omega(G)^c+1))$
   where $q$ is given by Lemma~\ref{Lemma: induced matching from non induced}. 
    Suppose the statement is false for this choice of $p$ for some $G$, $X$, and $b$. Let $I$ be an independent set of both $\Mca_{X,b}^1$ and $\Mca_{X,b}^2$ of size $p(\omega(G))$.
    By \cref{Lemma: induced matching from non induced}, there exists $I'\subseteq I$ such that $|I'|=2q(\omega(G),\omega^c+1)$
    and there is an induced matching $M_1$ that matches  $I'$ into $\delta_1(X)$. Let $Y$ be the vertices of $\delta_1(X)$ matched by  $M_1$.
\\
\\
    \sta{No vertex in $\delta_2(X)$ is adjacent to more than $2$ vertices of  $Y$ \label{claim: second layer of matching}}
    
    Suppose not and let $x_1,x_2,x_3\in Y$ and $v$ in $\delta_2(X)$ such that $\set{x_1,x_2,x_3}\subseteq N(v)$. Let $m_1,m_2,m_3$ be the vertices in $N(X)$ matched to respectively $x_1,x_2,x_3$ by $M_1$.
    Let us now consider $G' = G\sm N[X] \cup \set{x_1,x_2,x_3,m_1,m_2,m_3}$. Since  the set 
    $\{m_1,m_2,m_3\}$ is independent in $\Mca_{X,b}^2$, $x_1,x_2,x_3$ are in the same connected component of $G'$. Let $H$ be a minimal connected subgraph of $G'$ such that $\set{x_1,x_2,x_3}\subseteq N(H)$. We apply \cref{lem: minimalconnected} to analyze the structure of $H$.
    Since $x_1,x_2,x_3$ each have a unique neighbor in $H$ and $\set{m_1,m_2,m_3}$ is a stable set, the first outcome of \cref{lem: minimalconnected} cannot happen. If the second outcome happens, then $\{v\}\cup\bigcup_{i\leq 3} P_i$ forms a theta, which is a contradiction. Therefore, the last outcome of \cref{lem: minimalconnected} happens, but then $\{v\}\cup\bigcup_{i\leq 3} P_i$ forms a pyramid, which is also a contradiction.  This proves \eqref{claim: second layer of matching}.
\\
\\    
    Since every vertex of $\delta_1(X)$ has a neighbor in $\delta_2(X)$ and by \eqref{claim: second layer of matching}, we can find a (non-induced) matching of size $q(\omega(G),\omega(G)^c+1)$  that matches $Y$ into $\delta_2(X)$. Therefore, by Lemma~\ref{Lemma: induced matching from non induced}, there exists $Y'\subseteq Y$  with $|Y'| = \omega(G)^c+1$ and   an induced matching $M_2$ that matches $Y'$ into 
    to $\delta_2(X)$.  Using the edges incident with $Y'$ in both $M_1$ and $M_2$, we get a set of $\omega(G)^c+1$ pairwise anticomplete paths $P^1,\dots,P^{\omega(G)^c+1}$ of length $3$ from $N(X)$ to $\delta_2(X)$. Let $Z\subseteq I'$ be the set of the ends of these paths in $N(X)$.
\\
\\
    \sta{$Z$ is a constricted set in $(G\sm N[X]) \cup Z $. \label{claim: matching is constricted}}
    Suppose not; let $x_1,x_2,x_3 \in Z$ and let $T$ be an induced subgraph of $G$ such that
    $T$ is a tree and $x_1,x_2,x_3 \in T$. We may assume that $T$ is chosen to be minimal with these properties. Then either $T$ is a path with ends in $\{x_1,x_2,x_3\}$, or $T$ is a subdivision of the bipartite graph $K_{1,3}$ and $x_1,x_2,x_3$ are the leaves of $T$. 
    We may assume that $x_i$ is the end of the path $P^i$ for every $i \in \{1,2,3\}$.
    
    Let $G' = X\sm (\delta_1(X)\cup \delta_2(X)) \cup P^1 \cup P^2 \cup P^3$,
    and  let $H$ be a minimal connected subgraph of $G'$ such that $\set{x_1,x_2,x_3}\subseteq N(H)$. As before, we apply \cref{lem: minimalconnected} to analyze the structure of $H$.
    Since $x_1,x_2,x_3$ each have a unique neighbor in $H$ and since these neighbors are distinct and form a stable set, the first outcome of \cref{lem: minimalconnected} cannot happen. If the second outcome happens, then $T\cup\bigcup_{i\leq 3} P_i$ forms a theta, which is a contradiction. Therefore, the last outcome of \cref{lem: minimalconnected} happens, but then $T\cup\bigcup_{i\leq 3} P_i$ forms a pyramid, which is also a contradiction.  This proves \eqref{claim: matching is constricted}.
\\
\\  
    By Theorem~\ref{Theorem: small set hit all constricted path to v}, there are at most  
    $\omega(G)^c$ vertex disjoint paths from $Z$ to $b$ in  $G\sm X$, which contradicts the fact that $I'$ is an independent set of $\Mca_{X,b}^2$.
\end{proof}

We remind the reader of the celebrated Matroid Intersection Theorem \cite{edmonds1970}.
\begin{theorem}[Matroid Intersection Theorem~\cite{edmonds1970}]
	\label{thm:matroid_interesection}
 Let $M_1=(U,{\mathcal I}_1)$ and $M_2=(U,{\mathcal I}_2)$ be two matroids with the same ground set $U$. Then
 \[\max\,\{\,|I|\colon I\in {\mathcal I}_1\cap {\mathcal I}_2\,\}=\min\,\{\,\text{rank}_{M_1}(A)+\text{rank}_{M_2}(U\sm A)\big| A\subseteq U\,\}.\]
\end{theorem}

From \cref{lemma : small intersection of matroids} and  \cref{thm:matroid_interesection}, we deduce:

\begin{theorem}\label{corollary: matroid gives partition of neighborhood}
    There exists a polynomial $q$ such that the following holds.
    Let $G$ be a (theta, pyramid)-free graph, $X$ be a cooperative set in $G$, and $b\in V(G)$ anticomplete to $X$. Then there exists a partition $(A,N(X)\sm A)$ of $N(X)$ such that:
    \begin{itemize}
        \item the maximum matching from $A$ to $\delta(X_1)$ is of size at most $q(\omega(G))$ and
        \item there are a most $q(\omega(G))$ vertex-disjoint paths (except at $b$) from $N(X)\sm A$ to $b$ that are internally disjoint from $N[X]$.
    \end{itemize}
\end{theorem}

\section{Cooperative pairs and degenerate partitions} \label{sec: cooperative set evolving degenerate partition}
\MariaIsEditing{\textcolor{red}{DO NOT MODIFY! Maria has the token!}}
In this section, we prove \cref{thm: small pairwise separator omega version} and
deduce from it our first main result \cref{thm: small pairwise separator}.

Let $G$ be a graph, and let $X\subseteq G$ be connected and $C\subseteq G$ be such that $C\cap X = \mt$. We call a pair $(X,C)$ a \textit{cooperative pair in} $G$ if $X$ is cooperative in the connected component of $G\sm C$ that contains $X$. Let $G'\subseteq G$ and let $a$ and $b$ be distinct vertices of $G$. We say that a cooperative pair $(X,C)$ in $G'$ \textit{separates} $a$ from $b$ in $G$ if the following conditions hold:
\begin{itemize}
    \item  $a\in X\sm \delta^{G'}(X)$, and
    \item $b\in G\sm (X\cup C)$,  and 
    \item $\delta^{G'}(X)\cup C$ separates $a$ from $b$ in $G$.
\end{itemize}
In order to prove \cref{thm: small pairwise separator omega version}  we start by defining  a 
cooperative set $X_0=N_G^2(a)$, and 
create an improving sequence of pairs $(X_i,C_i)$ and subgraphs $G_i$ such that $(X_i,C_i)$ is a cooperative pair in $G_i$ that separates $a$ from $b$ in $G$. To make the notion of improvement precise, we define the ``value'' of a cooperative pair.  
To do so, we make use of a vertex partition for theta-free graphs first introduced in \cite{Three-path-configurations}.  This allows us to bound the number of improvement steps, and as a result, bound the size of the separator. We now explain this in detail.  

Following the proof of Theorem 7.1 of \cite{Three-path-configurations} and using Theorem 4.4 of \cite{MatijaPolyDegen} we deduce:

\begin{theorem} \label{degenerate partition}
  There exists a polynomial  $p$ such that the following holds:
  let $G$ be theta-free with $|V(G)|=n$ and clique number $\omega$. Then, there exists a partition
  $(S_1, \dots, S_k)$ of $V(G)$ with the following properties:
  \begin{enumerate}
  \item $k \leq p(\omega) \log n$.
  \item $S_i$ is a stable set for every $i \in \{1, \dots, k\}$.
  \item For every $i \in \{1, \dots, k\}$ and $v \in S_i$,

  we have
    $\deg_{G \setminus \bigcup_{j <i}S_j}(v)  \leq p(\omega)$. \label{hubsequence-3}
  \end{enumerate}
  \end{theorem}

Let $\Pi = (S_1, \dots, S_k)$ be a partition of $V(G)$ as in \cref{degenerate partition} and let and $v\in V(G)$. We define the \textit{value} of $v$ in $G$ denoted $val^G_\Pi(v)$ as the index $i$ such that $v\in S_i$. For a set $X\subseteq V(G)$, we define its \textit{value} in $G$, denoted $val^G_\Pi(X)$ as follows.  If $X \neq \emptyset$, then  
$val^G_\Pi(X)=\max\set{val^G_\Pi(v) \big| v \in\delta^G(X)}$
and if $X= \emptyset$,  then $val^G_\Pi(X) = 0$. 
Similarly, we define the \textit{value} of the cooperative pair in $G$ as $val^G_\Pi(X,C) = val^{G\sm C}_\Pi(X)$.

\begin{lemma}
    \label{lemma small neighborhood in cooperative}
    There exists a polynomial $p$ for which the following holds. Let $G$ be a (theta, pyramid)-free graph,  let $X$ be a cooperative set in $G$, and let $b\in N(X)$. Then  $|N(b)\cap X| \leq p(\omega(G))$.
\end{lemma}

\begin{proof}
    Suppose not and let $p(x)= x^3 p(x,3)> R(3,x+1)p(x,3)$ (by \cref{thm:ramsey}) where $p(x,y)$ is defined as in \cref{Lemma: induced matching from non induced}.
    Since no vertex in $\delta_2(X)$ has more than $R(3,\omega(G)+1)$ common neighbour with $b$ (as it would create a theta), and since every vertex in $\delta_1(X)$ has a neighbour in $\delta_2(X)$, there exist a matching from $N(b)\cap X$ to $\delta_2(X)$ of size $p(\omega(G),3)$. By \cref{Lemma: induced matching from non induced}, there exists an induced matching $M$ from $N(b)$ to $\delta_2(X)$ such that $|M|=3$. 
    Write $M=\{x_1y_1, x_2y_2, x_3y_3\}$ where 
     $x_1,x_2,x_3 \in \delta_1(X)$  and $y_1,y_2,y_3 \in \delta_2(X)$.
    Let $G' = X\sm(\delta_1(X)\cup\delta_2(X)) \cup \set{x_1,x_2,x_3,y_1,y_2,y_3}$ and let $H$ be a minimal connected subgraph of $G'$ such that $\set{x_1,x_2,x_3}\subseteq N(H)$. We apply \cref{lem: minimalconnected} to analyze the structure of $H$.
    Since each of $x_1,x_2,x_3$ has a unique neighbor in $H$, and since $\set{y_1,y_2,y_3}$ forms a stable set, the first outcome of \cref{lem: minimalconnected} cannot happen. If the second outcome happens, then $\{b\}\cup\bigcup_{i\leq 3} P_i$ forms a theta, which is a contradiction. Therefore, the last outcome of \cref{lem: minimalconnected} happens, but then $\{b\}\cup\bigcup_{i\leq 3} P_i$ forms a pyramid, which is also a contradiction.   
\end{proof}

We will need two classical results in graph theory:

\begin{theorem}[Menger's Theorem (Vertex Version)\cite{Menger1927}]
\label{thm: menger}
For any two nonadjacent sets of vertices $X$ and $Y$ in a finite graph $G$, 
the size of the smallest set of vertices whose removal separates $X$ from $Y$ 
is equal to the maximum number of pairwise internally vertex-disjoint paths between $X$ and $Y$.
\end{theorem}

\begin{theorem}[Kőnig's Theorem \cite{Konig}] 
\label{thm konig}
In any bipartite graph $G = (U, V, E)$, the size of a maximum matching 
is equal to the size of a minimum vertex cover.
\end{theorem}

The next lemma explains how to make one improvement step in the construction of a sequence of cooperative pairs. 
\begin{lemma} \label{lemma: improved cooperating pair} There exists a polynomial $p$ for which the following holds.
Let $G$ be a (theta, pyramid)-free graph, and let $a,b\in V(G)$. Let $\Pi$ be a partition of $V(G)$ with the properties given by Theorem~\ref{degenerate partition}.
Let $(X,C)$ be a cooperative pair separating $a$ and $b$ in $G$ such that $val_\Pi(X,C)> 1$. 
Then, there exists either
\begin{itemize}
    \item $G'\subseteq G$ and a cooperative pair $(X',C')$ in $G'$ separating $a$ and $b$ in $G$ such that $|C'|\leq |C|+p(\omega(G))$ and $val_\Pi^{G'}(X',C')< val^G_\Pi(X,C)$,  or
    
    \item  $C'\subseteq V(G)$ such that $C'$ separates $a$ from $b$ in $G$ and $|C'|\leq |C|+ \ p(\omega(G))$.
\end{itemize}
\end{lemma}

\begin{proof}

    Let $p_1$, $p_2$ and $p_3$ be the polynomials given by \cref{lemma small neighborhood in cooperative}, \cref{corollary: matroid gives partition of neighborhood} and \cref{degenerate partition}, respectively,  and let $p(x)=p_1(x)+2p_2(x)+p_2(x)p_3(x)$.
    Let $C_1= C \cup (N(b)\cap X)$. If $C_1$ separates $a$ from $b$ in $G$ we are done, so we may assume that there is a connected component $J$ of $G\sm(X\cup C_1)$ such that
    $J$  is not anticomplete to $X\sm C_1$ and such that $b \in J$. Let $X_1 = (X\sm \delta_1^{G\sm C}(X)) \cup N_{{G\sm C}}(J) \cap (X\sm C_1)$ and let $G_1 = N[J]\sm C_1 \cup X_1$. It follows from \cref{Lemma cooperative deletion} that $(X_1,C_1)$ is a cooperative pair in $G_1$.

    Let $\Mca_{X,b}^1$ and $\Mca_{X,b}^2$ on $G_1\sm C_1$ be defined as before.
    By \cref{corollary: matroid gives partition of neighborhood}, there exists a partition $(A,B)$ of $N_{G_1}(X_1)$ for which \begin{itemize}
        \item the maximum matching from $A$ to $\delta(X_1)$ is of size at most $p_2(\omega(G))$ and
        \item there are a most $p_2(\omega(G))$ vertex-disjoint paths (except at $b$) from $B$ to $b$ that are internally disjoint from $N[X]$.
    \end{itemize}
    By \cref{thm: menger}, there exists a set $K\subseteq G_1\sm(X_1\cup C_1)$ such that $|K|\leq p_2(\omega(G))$ and every path from $B$ to $b$ in $G_1\sm (X_1 \cup C_1 \cup K)$ intersects $A$. Let $C_2 = C_1 \cup K$. By \cref{thm konig}, there exists a set $F\subseteq N_{G_1}(X_1)\cup \delta^{G_1}(X_1)$ such that $|F|\leq q_2(\omega(G))$ and $A\sm F$ is anticomplete to $\delta^{G_1}(X_1)\sm F$.
    Let $F_1 = F\cap \delta^{G_1}(X_1)$ and $F_2 = F\cap A$ . Let $C_3 = C_2 \cup F_2$. Let \[H = \bigcup_{u\in F_1} \set{ v| v\in N(u) \text{ such that  } val_\Pi(u) < val_\Pi(v) }.\]
    Let $C' = C_3 \cup (H \cap N(X))$.
    By the third property  of the partition $\Pi$, we have that $$|C'|\leq |C_3| + |F_1| p_3(\omega(G)) \leq |C|+p_1(\omega(G))+ 2p_2(\omega(G)) + p_2(\omega(G))  p_3(\omega(G)) = |C|+p(\omega(G)).$$
    If $C'$ separates $a$ from $b$ in $G$, we are done. Therefore, we may assume that the connected component $D$ in $G\sm (C'\cup X_1)$ containing $b$ is not anticomplete to  $\delta^{G_1}(X_1)\sm C'$.
    Let $$X'' = X_1\sm\delta^{G_1}(X_1) \cup (N_{G\sm (C'\cup X_1)}(D) \cap X_1\sm C')$$ and $G'' = X'' \cup N_{G\sm (C'\cup X_1)}[D] \sm C'$.
    It follows from \cref{Lemma cooperative deletion} that $(X'',C')$ is a cooperative pair in $G''$.

    Let $\Gamma$ be the connected component of $G''\sm N_{G''}[X'']$ containing $b$ and let $X' = X''\cup N_{G''}(\Gamma)$.
    By \cref{Lemma N(cooperative) is cooperative}, we have that $(X',C')$ is a cooperative pair in $G' = X'' \cup N_{G''}[\Gamma]$.
    Moreover, we have that $\delta_1^{G'}(X') \cup C'$ separates $a$ from $b$ in $G$ since $N_{G\sm C'}(\Gamma) =\delta_1^{G'}(X')$.
    To conclude that $(X',C')$ and the subgraph $G'$ satisfy the first bullet in the statement of the theorem, it remains to show that $val^{G'}(X',C')< val^G(X,C)$.
    This follows since $\delta^{G'}(X')\subseteq N_G(X)\sm H$ and so every vertex in $\delta^{G'}(X')$ has a lower value than $val^G(X,C)$.
    \end{proof}

We can now prove the main result of this section.
\begin{theorem}\label{thm: small pairwise separator omega version}
    There exists a polynomial $p$ for which the following holds.
    Let $G$ be an $n$-vertex (theta, pyramid)-free graph, and $a,b$ be non-adjacent vertices in $G$. Then, there exist and $(a,b)$-separator $C\subseteq V(G)$ such that $|C| \leq p(\omega(G)) \log n$.
\end{theorem}

\begin{proof}
    Let $p_1$ and $p_2$ be defined as in \cref{degenerate partition} and \cref{lemma: improved cooperating pair}, respectively. Let $p(x) = x^6 +(p_1(x)+1)p_2(x)$ .
    Let $C_0 = N(a)\cap N(b)$, $G_0 = G\sm C_0$, and $X_0=N^2_{G_0}(a)$. Then $(X_0,C_0)$ is a cooperative pair separating $a$ from $b$ in $G$.
    Now let us recursively define a sequence $(X_i,C_i)_{i=0}^k$ where $(X_i,C_i)$ are the sets obtained by applying \cref{lemma: improved cooperating pair} to $(X_{i-1},C_{i-1})$ and $G$ as long as the first outcome of that lemma happens. Since the value of the cooperative pairs in this sequence is strictly decreasing, \cref{degenerate partition} implies that  $k\leq p_1(\omega(G)) \log(n)$. Applying \cref{lemma: improved cooperating pair} one more time gives us $C$ such that $C$ separates $a$ from $b$ in $G$ and $|C| \leq |C_0| + (k+1)\ p_2(\omega(G))$.

\sta{$|C_0| \leq R(3,\omega(G)+1)$. \label{claim: common neighbors}}
    Suppose not, then $N(a)\cap N(b)$ contains an independent set of size $3$, which together with $a$ and $b$ forms a theta.
    This proves \eqref{claim: common neighbors}.

    Therefore, using \cref{thm:ramsey}, we have that
    $|C| \leq \omega(G)^6 + (k+1)\ p_2(\omega(G)) \leq p(\omega(G))$ as required.
\end{proof}

 From \cref{thm: small pairwise separator omega version}, we deduce \cref{thm: small pairwise separator}, which we now restate.
\pairwisesep*
\begin{proof}
    Let $G\in \class$ and let $p$ be the polynomial from \cref{thm: small pairwise separator omega version}. Then, by \cref{thm: small pairwise separator omega version}, $G$ is $p(\omega(G))\log n$-pairwise separable. Since $p(\omega(G))\leq p(t) \leq t^c$ for a large enough $c$, the theorem follows.
\end{proof}

\section{Proofs of the main results}\label{sec: main results}

We have already proved  \cref{thm: small pairwise separator}; we prove
\cref{thm: small alpha pairwise separator} next. We make use of the following result from \cite{CESL}:
\begin{theorem}[Theorem 1.4 from \cite{CESL}]
\label{thm: from pairwise to alpha pairwise}
    For every positive integer $c$ there exists an integer $d=d(c)$ with the following property. If $\mathcal C$ is a hereditary graph class that is $(\omega(G) \log |V(G)|)^c$-pairwise separable, then for every $G \in \mathcal C$ on at least $3$ vertices and for every two non-adjacent vertices $u,v \in V(G)$, there exists a set $X \subseteq V(G)$ disjoint from $\{u,v\}$, with $\alpha(X) \leq \log^d (|V(G)|)$, that separates $u$ from $v$.
\end{theorem}

\begin{proof}[Proof of \cref{thm: small alpha pairwise separator}]
    It follows from \cref{thm: from pairwise to alpha pairwise} and \cref{thm: small pairwise separator omega version}.
\end{proof}

Next, we prove \cref{thm: log tw}. The structure of the proof is similar to
\cite{tw15}. We use the following:

\begin{lemma}[Lemma 8 of \cite{zbMATH06707229}]
\label{lemma:sep_tw}
Let $G$ be a graph. If for every weight function $w$, there exists a $w-$balanced separator $X$ such that $|X|\leq d$, then $\tw(G) \leq 2 d$.
\end{lemma}

\begin{theorem}[Corollary of Theorem 9.2 from \cite{tw15}]
\label{bettersep}
    Let $d,L \in \nat$.
  Let $G$ be $ L-$pairwise separable and let $w$ be a weight function on $G$.  If $G$ is $d$-breakable then there exists a $w$-balanced separator $X$ in $G$ such that $|X|\leq 3Ld$.
\end{theorem}

We are now ready to prove \cref{thm: log tw}.
\begin{proof}
    Every graph $G$ is $K_{\omega(G)+1}-$free.
    Let $c'$ be such that $p(x+1)\leq x^{c'}$ where $p$ is defined as in \cref{thm: small pairwise separator omega version}, and let $d=d(t)$ be defined as in \cref{thm:domsep}.
    By \cref{thm: small pairwise separator}, $G$ is $\omega(G)^{c'}\log n$-pairwise separable. By \cref{thm:domsep}, $G$ is $d$-breakable. Therefore, by \cref{bettersep}, 
    for every weight function $\mu$ 
    there exist a $mu$-balanced separator $X$ of $G$ such that $|X|\leq 3d\omega(G)^{c'}\log n$. By \cref{lemma:sep_tw} this  implies that $\tw(G)\leq 6d\omega(G)^{c'}\log n$. Taking $c$ large completes the proof (in fact, $c=c'+\log(6d)$ is enough).
\end{proof}

Finally, we prove  \cref{thm: small tree alpha}. We  use the following  result from
\cite{CESL}:
\begin{theorem}[Theorem 1.1 from \cite{CESL}]\label{thm: from polylog w to polylog}
    Let $\mathcal{C}$ be a hereditary graph class. The following are equivalent:
    \begin{enumerate}\setlength\itemsep{-.7pt}
        \item[(i)] There exists an integer $c_1 > 0$ such that for every $G \in \mathcal{C}$ on at least $3$ vertices we have $\atw(G) \leq (\log |V(G)|)^{c_1}$
        \item[(ii)] There exists  an integer $c_3 > 0$ such that for every $G \in \mathcal{C}$ on at least $3$ vertices we have $\tw(G) \leq (\omega(G) \log |V(G)|)^{c_3}$
    \end{enumerate}
\end{theorem}

\begin{proof}[Proof of \cref{thm: small tree alpha}]
    It follows from \cref{thm: log tw} and \cref{thm: from polylog w to polylog}.
\end{proof}

\section{Acknowledgment}

We are grateful to Pawe\l{} Rz\k{a}\.zewski for allowing us to use ~\cref{fig:ex esd}.
We also thank Sepehr Hajebi and Sophie Spirkl for many helpful discussions.

\bibliographystyle{abbrv}
\bibliography{ref}


\end{document}

%% file: graph.tikzstyles
\usetikzlibrary {decorations.pathreplacing}

\tikzstyle{box}=[shape=rectangle, text height=1.5ex, text depth=0.25ex, yshift=0.5mm, fill=white, draw=black, minimum height=5mm, yshift=-0.5mm, minimum width=5mm, font={\small}]
\tikzstyle{gate}=[shape=rectangle, text height=1.5ex, text depth=0.25ex, yshift=0.5mm, fill=white, draw=black, minimum height=5mm, yshift=-0.5mm, minimum width=5mm, font={\small}, tikzit category=circuit]
\tikzstyle{big gate}=[shape=rectangle, text height=1.5ex, text depth=0.25ex, yshift=0.5mm, fill=white, draw=black, minimum height=10mm, yshift=-0.5mm, minimum width=5mm, font={\small}, tikzit category=circuit]
\tikzstyle{Z dot}=[inner sep=0mm, minimum size=2mm, shape=circle, draw=black, fill={rgb,255: red,221; green,255; blue,221}, tikzit category=zx]
\tikzstyle{Z phase dot}=[minimum size=5mm, font={\footnotesize\boldmath}, shape=rectangle, rounded corners=2mm, inner sep=0.2mm, outer sep=-2mm, scale=0.8, tikzit shape=circle, draw=black, fill=white, tikzit draw=blue, tikzit category=zx]
\tikzstyle{X dot}=[Z dot, shape=circle, draw=black, fill={rgb,255: red,255; green,136; blue,136}, tikzit category=zx]
\tikzstyle{X phase dot}=[Z phase dot, tikzit shape=circle, tikzit draw=blue, fill={rgb,255: red,255; green,136; blue,136}, font={\footnotesize\boldmath}, tikzit category=zx]
\tikzstyle{hadamard}=[fill=yellow, draw=black, shape=rectangle, inner sep=0.6mm, minimum height=1.5mm, minimum width=1.5mm, tikzit category=zx]
\tikzstyle{paulibox}=[fill={rgb,255: red,221; green,221; blue,255}, draw=black, shape=rectangle, inner sep=0.6mm, minimum height=5mm, minimum width=5mm, font={\footnotesize}, text height=1.5ex, text depth=0.25ex, tikzit category=zx]
\tikzstyle{vertex}=[inner sep=0mm, minimum size=1mm, shape=circle, draw=black, fill=black, tikzit category=misc]
\tikzstyle{vertex set}=[inner sep=0mm, minimum size=2mm, shape=circle, draw=black, fill=white, font={\footnotesize\boldmath}, tikzit category=misc]
\tikzstyle{small black dot}=[fill=black, draw=black, shape=circle, inner sep=0pt, minimum width=1.2mm, tikzit category=circuit]
\tikzstyle{cnot ctrl}=[fill=black, draw=black, shape=circle, inner sep=0pt, minimum width=1.2mm, tikzit category=circuit]
\tikzstyle{cnot targ}=[fill=white, draw=white, shape=circle, tikzit category=circuit, label={center:$\oplus$}, inner sep=0pt, minimum width=2.1mm, tikzit fill={rgb,255: red,102; green,204; blue,255}, tikzit draw=black]
\tikzstyle{ket}=[fill=white, draw=black, shape=regular polygon, regular polygon sides=3, regular polygon rotate=-30, scale=0.7, inner sep=1pt, tikzit category=circuit, tikzit shape=rectangle, tikzit fill=green]
\tikzstyle{bra}=[fill=white, draw=black, shape=regular polygon, regular polygon sides=3, regular polygon rotate=30, scale=0.7, inner sep=1pt, tikzit category=circuit, tikzit shape=rectangle, tikzit fill=red]
\tikzstyle{scalar}=[shape=rectangle, text height=1.5ex, text depth=0.25ex, yshift=0.5mm, fill=white, draw=black, minimum height=5mm, yshift=-0.5mm, minimum width=5mm, font={\small}]
\tikzstyle{clabel}=[fill=white, draw=none, shape=rectangle, tikzit fill={rgb,255: red,56; green,255; blue,242}, font={\footnotesize}, inner sep=1pt, tikzit category=labels]
\tikzstyle{empty diagram}=[draw={gray!40!white}, dashed, shape=rectangle, minimum width=1cm, minimum height=1cm, tikzit category=misc]
\tikzstyle{white dot}=[Z dot]
\tikzstyle{gray dot}=[X dot]
\tikzstyle{white phase dot}=[Z phase dot]
\tikzstyle{gray phase dot}=[X phase dot]
\tikzstyle{small hadamard}=[hadamard]

\tikzstyle{simple}=[-]
\tikzstyle{hadamard edge}=[-, dashed, dash pattern=on 2pt off 0.5pt, thick, draw={rgb,255: red,68; green,136; blue,255}]
\tikzstyle{non edge}=[-, dashed, dash pattern=on 2pt off 1pt, draw=black]
\tikzstyle{box edge}=[-, dashed, dash pattern=on 2pt off 0.5pt, thick, draw={rgb,255: red,203; green,192; blue,225}]
\tikzstyle{brace edge}=[-, tikzit draw=blue, decorate, decoration={brace,amplitude=1mm,raise=-1mm}]
\tikzstyle{diredge}=[->]
\tikzstyle{double edge}=[-, double, shorten <=-1mm, shorten >=-1mm, double distance=2pt]
\tikzstyle{gray edge}=[-, {gray!60!white}]
\tikzstyle{pointer edge}=[->, very thick, gray]
\tikzstyle{boldedge}=[-, line width=1.6pt, shorten <=-0.17mm, shorten >=-0.17mm]

%% file: commands.tex
\usepackage{bm}
\usepackage{bbm}
\usepackage{thmtools}

\newcommand{\set}[1]{\left\{#1\right\}}
\newcommand{\sm}{\setminus}
\newcommand{\mt}{\emptyset}

\newcommand{\nat}{{\mathbb N}}



%% file: esd_definition.tex
\MariaIsEditing{\textcolor{red}{DO NOT MODIFY! Maria has the token!}}
An important tool in the proof of Theorem~\ref{thm: small pairwise separator} is the ``extended strip decompositions'' of \cite{Threeinatree}. We explain this now after introducing some definitions from \cite{TIV}.
Let $G,H$ be graphs, and let $Z \subseteq V(G)$. Let $W$ be the set of vertices of degree one in $H$. Let $T(H)$ be the set of all triangles of $H$.
Let $\eta$ be a map with domain
the union of $E(H)$, $V(H)$, $T(H)$,  and the set of all pairs $(e, v)$ where $e \in E(H)$, $v \in V (H)$ and $e$ incident
with $v$, and range $2^{V(G)}$,  satisfying the following conditions:
\begin{itemize}
\item For every $v \in V(G)$ there exists a unique
$x \in E(H)  \cup V(H) \cup T(H)$ such that $v \in \eta(x)$.
\item For every $e \in E(H)$ and $v \in V(H)$ such that $e$ is incident with $v$, $\eta(e,v) \subseteq \eta(e)$
\item Let  $e,f \in E(H)$ with  $e \neq f$,  and $x \in \eta(e)$ and $y \in \eta(f)$. Then $xy \in E(G)$ if and only if
$e, f$ share an end-vertex $v$ in $H$, and $x \in \eta(e, v)$ and
$y \in  \eta(f, v)$.
\item If $v \in V (H)$, $x  \in \eta (v)$, $y \in V (G) \setminus \eta(v)$, and
$xy \in E(G)$, then $y \in \eta(e, v)$ for some
$e \in E(H)$ incident with $v$.
\item  If $D \in T(H)$,  $x \in \eta (D)$, $y \in V (G) \setminus \eta(D)$
and $xy \in E(G)$, then
$y \in \eta(e,u) \cap \eta(e,v)$  for some distinct $u, v \in D$, where
$e$ is the edge $uv$ of $H$.
\item  $|Z| = |W |$, and for each $z \in Z$ there is a vertex $w \in W$ such that $\eta(e, w) = \{z\}$, where $e$ is the 
(unique) edge of $H$ incident with $w$.
\end{itemize}
Under these circumstances, we say that  $\eta$ is an {\em extended strip decomposition of
$(G,Z)$ with pattern $H$} (see Figure~\ref{fig:ex esd}).
As a slight abuse of notation, for $v\in V(G)$ we will denote by $\eta^{-1}(v)$ the unique $x \in E(H)  \cup V(H) \cup T(H)$ such that $v \in \eta(x)$, as guaranteed by the first condition.

Let  $e$ be an edge of $H$ with ends $u,v$. An {\em $e$-rung} in
$\eta$ is a path $p_1 \dd \dots \dd p_k$ (possibly $k=1)$ in $\eta(e)$,
with $p_1 \in \eta(e,v)$, $p_k \in \eta(e,u)$ and
$\{p_2, \dots, p_{k-1}\} \subseteq \eta(e) \setminus (\eta(e,v)  \cup \eta(e,u))$. We say that $\eta$ is  {\em faithful} if for every $e \in E(H)$,
there is an $e$-rung in $\eta$.

A set $A \subseteq V(G)$ is an {\em atom} of $\eta$ if one of the following holds:
\begin{itemize}
\item $A=\eta(v)$ for some $v \in V(H) $.
\item $A = \eta(D)$ for some $D \in T(H)$.
\item $A=\eta(e) \setminus (\eta(e,u) \cup \eta(e,v))$ for some
edge $e$ of $H$ with ends $u,v$.
\end{itemize}
We say that an atom is a \textit{vertex atom}, a \textit{triangle atom}, or an \textit{edge atom} depending on which of the previous conditions holds.
For an atom $A$ of $\eta$, the {\em boundary} $\delta(A)$ of $A$ is defined as follows:
\begin{itemize}
\item If $v \in V(H)$ and  $A=\eta(v)$,  then
$\delta(A)=\bigcup_{e \in E(H) \; \colon \; e \text{ is incident with } v} \eta(e,v)$.
\item If $A=\eta(D)$ , and $D \in T(H)$ with $D=v_1v_2v_3$, then
$\delta(A)=\bigcup_{i \neq j \in \{1,2,3\}}\eta(v_iv_j,v_i) \cap \eta(v_iv_j,v_j)$
\item If $A=\eta(e) \setminus (\eta(e,u) \cup \eta(e,v))$ for some
edge $e$ of $H$ with ends $u,v$, then $\delta(A)=\eta(e,u) \cup \eta(e,v)$.
\end{itemize}
A set $Z \subseteq V (G)$ is {\em constricted} if it is stable and for every $T \subseteq G$ such that $T$ is a tree,  $|Z \cap V (T )| \leq  2$.

\begin{figure}[h]
    \centering
\includegraphics[width=0.8\linewidth]{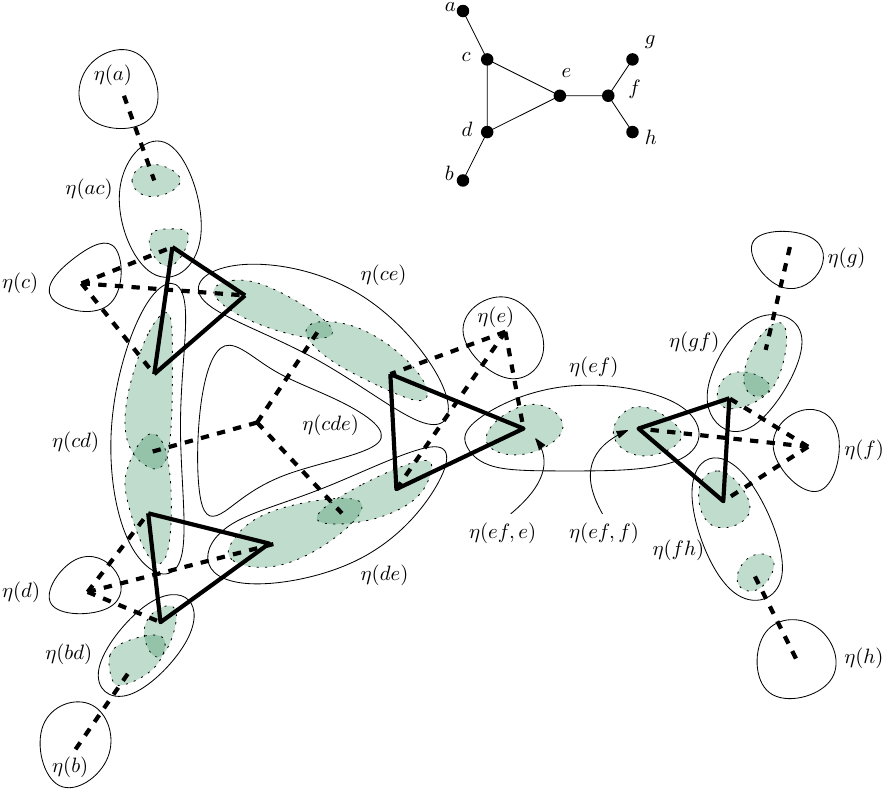}
    \caption{Example of an extended strip decomposition with its pattern (here dash lines represent potential edges).
    This figure was created by Pawe\l{} Rz\k{a}\.zewski and we use it with his permission.}
    \label{fig:ex esd}
\end{figure}

The main result of \cite{Threeinatree} is the following.
\begin{theorem}\label{stripdecomp}
Let $G$ be a connected graph and let $Z \subseteq  V (G)$ with $|Z| \geq 2$.
Then $Z$ is constricted if and only
if for some graph $H$, $(G, Z)$ admits a faithful extended strip decomposition with pattern $H$.
\end{theorem}

%% file: figures/dbs-2.tikz
\begin{tikzpicture}
	\begin{pgfonlayer}{nodelayer}
		\node [style=none] (0) at (-6, 0) {};
		\node [style=none] (1) at (6, 0) {};
		\node [style=vertex] (2) at (0, 0) {};
		\node [style=none] (3) at (-0.5, 0.5) {};
		\node [style=none] (4) at (0.5, 0.5) {};
		\node [style=none] (6) at (0.15, 0.925) {\tiny $\small  X'$};
		\node [style=vertex] (7) at (-3.75, -3.5) {};
		\node [style=none] (8) at (-4, 0) {};
		\node [style=none] (9) at (3.25, 0) {};
		\node [style=none] (10) at (-4.5, -3.5) {\tiny $n'$};
		\node [style=vertex] (11) at (2.75, -3.75) {};
		\node [style=none] (12) at (4.5, 0) {};
		\node [style=none] (13) at (-2.5, 0) {};
		\node [style=none] (14) at (-4.5, -5.5) {};
		\node [style=none] (15) at (-2, -4.25) {};
		\node [style=none] (16) at (1.5, -4.25) {};
		\node [style=none] (17) at (2, -7) {};
		\node [style=none] (18) at (-1.5, -7.25) {};
		\node [style=none] (19) at (-3.75, -7) {};
		\node [style=none] (20) at (3.25, -5.25) {};
		\node [style=none] (21) at (-0.5, -5.5) {$\huge B$};
		\node [style=none] (22) at (-2.75, -5.75) {};
		\node [style=none] (23) at (1.5, -5.5) {};
		\node [style=none] (24) at (3.5, -3.5) {\tiny $n$};
		\node [style=none] (25) at (3.25, 0.5) {};
		\node [style=none] (26) at (4.75, 0.5) {};
		\node [style=none] (27) at (4, 0.925) {};
		\node [style=none] (28) at (4, 0.925) {\tiny $\small  P_2$};
		\node [style=none] (29) at (-4, 0.5) {};
		\node [style=none] (30) at (-2.5, 0.5) {};
		\node [style=none] (31) at (-3.25, 0.925) {};
		\node [style=none] (32) at (-3.25, 0.925) {\tiny $\small P_1$};
		\node [style=none] (33) at (0.25, -6.75) {\tiny $ P_3$};
	\end{pgfonlayer}
	\begin{pgfonlayer}{edgelayer}
		\draw (1.center) to (0.center);
		\draw [style=brace edge] (3.center) to (4.center);
		\draw (7) to (8.center);
		\draw (7) to (9.center);
		\draw (13.center) to (11);
		\draw (11) to (12.center);
		\draw (15.center) to (16.center);
		\draw [in=90, out=0] (16.center) to (20.center);
		\draw [in=60, out=-90, looseness=1.25] (20.center) to (17.center);
		\draw [in=0, out=-135] (17.center) to (18.center);
		\draw [in=-15, out=180] (18.center) to (19.center);
		\draw [in=-105, out=165, looseness=1.25] (19.center) to (14.center);
		\draw [in=180, out=90, looseness=0.75] (14.center) to (15.center);
		\draw [in=90, out=-105, looseness=1.50] (7) to (22.center);
		\draw [in=-120, out=-90, looseness=0.75] (22.center) to (23.center);
		\draw [in=-60, out=60, looseness=1.50] (23.center) to (11);
		\draw [style=brace edge] (25.center) to (26.center);
		\draw [style=brace edge] (29.center) to (30.center);
	\end{pgfonlayer}
\end{tikzpicture}

%% file: figures/dbs-5.tikz
\begin{tikzpicture}
	\begin{pgfonlayer}{nodelayer}
		\node [style=none] (0) at (0.5, -0.5) {};
		\node [style=none] (1) at (12.5, -0.5) {};
		\node [style=none] (7) at (2.5, -0.5) {};
		\node [style=none] (8) at (9.75, -0.5) {};
		\node [style=none] (11) at (11, -0.5) {};
		\node [style=none] (12) at (4, -0.5) {};
		\node [style=none] (13) at (2, -5.5) {};
		\node [style=none] (14) at (4.5, -4.25) {};
		\node [style=none] (15) at (8, -4.25) {};
		\node [style=none] (16) at (8.5, -7) {};
		\node [style=none] (17) at (5, -7.25) {};
		\node [style=none] (18) at (2.75, -7) {};
		\node [style=none] (19) at (9.75, -5.25) {};
		\node [style=none] (20) at (6, -5.75) {$\huge B$};
		\node [style=vertex] (21) at (2.5, -4) {};
		\node [style=vertex] (22) at (4.25, -2) {};
		\node [style=vertex] (23) at (5.75, -0.5) {};
		\node [style=none] (24) at (8.5, -0.5) {};
		\node [style=none] (26) at (3.25, -5.5) {};
		\node [style=none] (28) at (3, -0.5) {};
		\node [style=vertex] (29) at (6, -2.75) {};
		\node [style=none] (30) at (4.5, -0.5) {};
		\node [style=none] (31) at (7.75, -0.5) {};
		\node [style=none] (32) at (5.5, -2.75) {\tiny $k_i$};
		\node [style=none] (33) at (3.75, -2) {\tiny $q_1$};
		\node [style=none] (34) at (2, -4) {\tiny  $q_s$};
		\node [style=none] (35) at (1.5, -2.25) {\footnotesize $Q$};
		\node [style=none] (36) at (5.25, -3.75) {\footnotesize $R$};
		\node [style=none] (37) at (5.75, -0.175) {\tiny $v$};
		\node [style=none] (38) at (4.5, 0.25) {};
		\node [style=none] (39) at (7.75, 0.25) {};
		\node [style=none] (40) at (6.3, 0.55) {{\footnotesize $P'$}};
	\end{pgfonlayer}
	\begin{pgfonlayer}{edgelayer}
		\draw (1.center) to (0.center);
		\draw (14.center) to (15.center);
		\draw [in=90, out=0] (15.center) to (19.center);
		\draw [in=60, out=-90, looseness=1.25] (19.center) to (16.center);
		\draw [in=0, out=-135] (16.center) to (17.center);
		\draw [in=-15, out=180] (17.center) to (18.center);
		\draw [in=-105, out=165, looseness=1.25] (18.center) to (13.center);
		\draw [in=180, out=90, looseness=0.75] (13.center) to (14.center);
		\draw [in=120, out=-90, looseness=0.50] (21) to (26.center);
		\draw [in=-165, out=90] (21) to (7.center);
		\draw [in=120, out=-15] (28.center) to (22);
		\draw (22) to (23);
		\draw (12.center) to (29);
		\draw (29) to (30.center);
		\draw (29) to (31.center);
		\draw (24.center) to (29);
		\draw [in=-45, out=-105, looseness=1.25] (29) to (26.center);
		\draw [style=brace edge] (38.center) to (39.center);
	\end{pgfonlayer}
\end{tikzpicture}

%% file: esd_section.tex
Although we only use the results of this section on graphs in $\class$, they hold for a much larger class, namely $K_{t,t}$-free graphs. We, therefore, prove them in their full generality.
Let $G$ and $H$ be graphs, $Z\subseteq V(G)$ with $|Z|\geq 2$ and let $\eta$ be an extended strip decomposition of $(G,Z)$ with pattern $H$.  We say that a path $P=p_1\dd\dots\dd p_k$ in $G$ is a \textit{leaf path starting at $p_1$ and ending at $p_k$} if $p_1 \in Z$. 


The main result of this section is the following:

\begin{restatable}{theorem}{alphaHitPaths}
\label{lemma: hit all paths to v alpha version}
For every $t \in \nat$ there exists  $c=c(t)\in \nat$ with the following property.
    Let $G$ be a $K_{t,t}$-free graph, $Z \subseteq  V(G)$ with $|Z| \geq 2$ and $H$ be a graph. Let $\eta$ be an extended strip decomposition of $(G,Z)$ with pattern $H$. Then for every vertex $v\in V(G)\sm Z$, there exists $X\subseteq V(G)\sm v$ such that $X$ intersects all the leaf paths ending at $v$ and $\alpha(X)\leq c(t)$.
\end{restatable}

In the remainder of the proof, we  will need the following consequence of
\cref{lemma: hit all paths to v alpha version}:

\begin{theorem}\label{Theorem: small set hit all constricted path to v}
 For every $t \in \nat$ there exists  $c=c(t)\in \nat$ with the following property.
    Let $G$ be a $K_{t,t}$-free graph, let $\omega=\omega(G)$, and let $Z \subseteq  V (G)$ be constricted. Then for every vertex $v\in V(G)\sm Z$, there exist at most $\omega^c$ pairwise vertex-disjoint (except at $v$) paths starting in $Z$ and ending at $v$. 
\end{theorem}

We first prove \cref{Theorem: small set hit all constricted path to v}
assuming \cref{lemma: hit all paths to v alpha version}.
For positive integers $a,b$ let $R(a,b)$ be the smallest integer $R$, often called the Ramsey number, such that every graph on $R$ vertices contains either a stable set of size $a$ or a clique of size $b$.
\begin{theorem}[Ramsey \cite{ramsey}]\label{thm:ramsey}
    For all $c,s\in \nat$, $R(c,s) \leq c^{s-1}$
\end{theorem}

\begin{lemma}\label{lemma: hit all paths esd version}
    For every $t \in \nat$ there exists  $c=c(t)\in \nat$ with the following property.
    Let $G$ be a $K_{t,t}$-free graph, $\omega=\omega(G)$, $Z \subseteq  V (G)$ with $|Z| \geq 2$. Let $H$ be a graph, and let $\eta$ be an extended strip decomposition of $(G,Z)$ with pattern $H$. Then, for every vertex $v\in V(G)\sm Z$, there are at most $\omega^c$ pairwise vertex-disjoint (except at $v$) leaf paths ending at $v$. 
\end{lemma}
\begin{proof}
    Let $c(t)$ be the constant given by \cref{lemma: hit all paths to v alpha version}.
    By Lemma~\ref{lemma: hit all paths to v alpha version}, there exists a set $X$ that intersects all the leaf paths ending in $v$ and $\alpha(X)\leq c(t)$. Since $G$ is $K_{\omega+1}$-free, we have that $|X|\leq R(c(t),\omega+1) \leq \omega^{2c(t)}$,
    and the result follows.
\end{proof}

\begin{proof}[Proof of Theorem~\ref{Theorem: small set hit all constricted path to v}]
By Theorem~\ref{stripdecomp}, there exists an extended strip decomposition of $(G,Z)$. Now, the theorem follows from Lemma~\ref{lemma: hit all paths esd version}.
\end{proof}

We now turn to the proof of \cref{lemma: hit all paths to v alpha version}.
For the rest of this section, fix an integer $t>1$ and let $G$ be a $K_{t,t}$-free graph, $Z \subseteq  V (G)$ with $|Z| \geq 2$, and $H$ be a graph. Let $\eta$ be an extended strip decomposition of $(G,Z)$ with pattern $H$.
Let $x,y\in V(H)$, we say that $y$ is a \textit{special neighbor} of $x$ if $xy\in E(H)$ and $\alpha(\eta(xy,x))\geq t$. In that case, we call $xy$ a {\em special edge} of $x$. We say that a vertex $v \in V(G)$ is {\em safe}
if either
\begin{itemize}
\item there is no edge $xy$ of $H$ such that $v \in \eta(xy,x)$, or
\item there exists an edge $xy$ of $H$ such that $y$ is a
special neighbor of $x$, and $v \in \eta(xy,x)$.
\end{itemize}
Let $Safe(G)$ denote the set of all safe vertices of $G$. We observe:

\begin{lemma}\label{lemma: easy alpha emulation one vertex}
    Let $G$ be a $K_{t,t}$-free graph, and let $X_1,\dots,X_k \subseteq V(G)$ be disjoint and pairwise complete, then either
    \begin{itemize}
        \item $\alpha\left(\bigcup_{i=1}^n X_i\right) < t$, or
        \item there exists $i ^*$ such that $\alpha(X_{i^*}) \geq t$ and $\alpha\left(\bigcup_{i\neq i^*} X_i\right) < t$
    \end{itemize}
    Consequently, every vertex of $H$ has at most one special neighbor.
    \end{lemma}

\begin{proof}
    Notice that any stable set in $\bigcup_{i\neq i} X_i$ is a subset of one of the $X_i$. Therefore, the first case holds if there is no $i^*$ such that $\alpha(X_i)\geq t$. So we can assume that such an $i^*$ exists.
    Since $G$ is $K_{t,t}$-free, $i^*$ is unique. Moreover, since $X_{i^*}$ is complete to $\bigcup_{i\neq i^*} X_i$, we have that $\alpha\left(\bigcup_{i\neq i^*} X_i\right) < t$.
\end{proof}

For a vertex $x \in V(H)$, let $Em(x)=\bigcup \eta(xy,x)$
where the union is taken over all neighbors $y$ of $x$ that are not
special. We call $Em(x)$ {\em the set emulating $x$} in $G$.
Observe that $Em(x) \cap Safe(G)=\emptyset$.
It follows immediately from \cref{lemma: easy alpha emulation one vertex}
that $\alpha(Em(x)) < t$ for every $x \in V(H)$.

We are now ready to prove \cref{lemma: hit all paths to v alpha version}.
\begin{proof}
Let $c=4t+6$  and suppose that $v \in V(G)$ violates the conclusion of
\cref{lemma: hit all paths to v alpha version}.

Let $A$ be an atom of $\eta$. If $A=\eta(x)$ is a vertex atom, we say that
$A$ {\em points to an edge $e$ of $H$} if $e=xy$ and
$y$ is a special neighbor of $x$.
If $A=\eta(x_1x_2x_3)$ is a triangle  atom, we say that
$A$ {\em points to  an edge $e$ of $H$} if $e=x_ix_j$ for some $i,j \in \{1,2,3\}$ and  $x_i$ is a special neighbor of $x_j$, and $x_j$ is a special neighbor of
$x_i$.

\sta{Every vertex atom and every triangle atom points to at most one edge. \label{uniquepoint}}

Suppose first that  $A=\eta(x)$ is a vertex atom. Then $A$ only points to edges incident with $x$. Let $y \in V(H)$ be such that  $A$ points to the edge $xy$. Then
$y$ is a special neighbor of $x$, and therefore $xy$ is unique by
\cref{lemma: easy alpha emulation one vertex}.

Thus, we may assume that $A=\eta(x_1x_2x_3)$ is a triangle atom.
We may assume that $A$ points to the edge $x_1x_2$. Then
  $x_2$ is a special neighbor
of $x_1$, and $x_1$ is a special neighbor of $x_2$.
By  \cref{lemma: easy alpha emulation one vertex}
$x_3$ is not a special neighbor of $x_1$, and $x_3$ is not a special neighbor
of $x_2$.  Consequently,  $A$ does not point at
$x_1x_3$, and $A$ does not point at $x_2x_3$.
This proves \eqref{uniquepoint}.
\\
\\
Let $A= \eta(x)$ be a vertex atom  of  $\eta$.
If $A$ does not point to any edge, let $\overline{A}=A$ and let
$\Delta(A)=Em(x)$.
If $A$ points to an  edge $xy$ (which is unique by \eqref{uniquepoint}),
let $\overline{A}$ be the union of $\eta(xy) \cap Safe(G)$
with all the vertex and triangle atoms that point to $xy$,
and let $\Delta(A)=Em(x) \cup Em(y)$.

\sta{If $A$ is a vertex atom, then $\alpha(\Delta(A)) <2t$,
and $\Delta(A)$ separates $\overline{A}$ from $V(G) \setminus (\overline{A} \cup \Delta(A))$. \label{vertexatom}}

Since $\Delta(A)$ is contained in the union of at most two sets emulating
a vertex of $H$, it follows that 
$\alpha(\Delta(A)) < 2t$.
The second statement of \eqref{vertexatom} follows from
the definition of a strip structure. This proves \eqref{vertexatom}.

\sta{There is no $x \in V(H)$ such that $v \in \eta(x)$. \label{notvertex}}

Suppose $v \in A=\eta(x)$ for some $x \in V(H)$. Let
$X=\Delta(A) \cup (\overline{A} \cap Z)$. Since there is at
most one edge $e$ of $H$ such that $\overline{A}$ meets the
set $\eta(e)$, it follows that $|\overline{A} \cap Z| \leq 2$,
and so $\alpha(X) < 2t+2$. Since by \eqref{vertexatom} $X$ meets
every leaf path ending at $v$ that does not start in $\overline{A}$, 
\eqref{notvertex} follows.
\\
\\
Let $A= \eta(x_1x_2x_3)$ be a triangle atom  of $\eta$.
If $A$ does not point to any edge, let $\overline{A}=A$
and let $\Delta(A)=Em(x_1) \cup Em(x_2) \cup Em(x_3)$.
If $A$ points to an  edge $x_ix_j$ (which is unique by \eqref{uniquepoint}),
let $\overline{A}$ be the union of $\eta(x_ix_j) \cap Safe(G)$ 
with all the vertex and triangle atoms that point to $x_ix_j$, and
let $\Delta(A)= Em(x_i) \cup Em(x_j)$.

\sta{If $A$ is a triangle  atom, then $\alpha(\Delta(A)) <3t$, and
$\Delta(A)$ separates $\overline{A}$ from $V(G) \setminus (\overline{A} \cup \Delta(A))$. \label{triangleatom}}

Since $\Delta(A)$ is contained in the union of at most three sets emulating
a vertex of $H$, it follows that 
$\alpha(\Delta(A)) < 3t$.
The second statement of \eqref{triangleatom} follows from
the definition of a strip structure. This proves \eqref{triangleatom}.

\sta{There is no triangle $x_1x_2x_3$ of $H$  such that $v \in \eta(x_1x_2x_3)$. \label{nottriangle}}

Suppose $v \in A=\eta(x_1x_2x_3)$ for some triangle $x_1x_2x_3$ of $H$. Let
$X=\Delta(A) \cup (\overline{A} \cap Z)$. Since there is at
most one edge $e$ of $H$ such that $\overline{A}$ meets the set $\eta(e)$, it follows that $|\overline{A} \cap Z| \leq 2$,
and so $\alpha(X) < 2t+2$. Since by \eqref{triangleatom} $X$ meets
every leaf path ending at $v$ that does not start in $\overline{A}$, 
\eqref{nottriangle} follows.
\\
\\
\sta{$v \not \in Safe(G)$. \label{notsafe}}

Suppose that $v \in Safe(G)$. By \eqref{notvertex} and
\eqref{nottriangle}, $v \in  \eta(xy) \cap Safe(G)$ for some
edge $xy$ of $H$. Let $X=Em(x) \cup Em(y) \cup (\eta(xy) \cap Z)$.
Then $v \not \in X$. Since $|\eta(xy) \cap Z| \leq 2$,
we deduce   $\alpha(X) < 2t+2$.
It follows from
the defintion of a strip structure that $Em(x) \cup Em(y)$ meets
every leaf path ending at $v$ that does not start in
$\eta(xy) \cap Safe(G)$, and \eqref{notsafe} follows.
\\
\\
By \eqref{notsafe}, there exists an edge $xy$  of $H$ such that
$y$ is not a special neighbor of $x$ and $v \in \eta(xy,x)$.
Moreover, if $v \in \eta(xy,y)$, then $x$ is not a special neighbor
of $y$. Let $x'$ be the special neighbor of $x$ (if one exists),
and let $y'$ be the special neighbor of $y$ (if one exists). Note that
$x' \neq y$, but possibly $y'=x$.
Suppose first that either 
\begin{itemize}
    \item  $x=y'$, or 
    \item $x \neq y'$ and $v \in \eta(xy,y)$.
    \end{itemize}
Let $\mathcal{A}$ be the union of the following sets
\begin{itemize}
\item $\{v\}$,
\item $\eta(xy) \cap Safe(G)$,
\item $\eta(xx') \cap Safe(G)$ (if $x'$ is defined),
\item $\eta(yy') \cap Safe(G)$ (if $y'$ is defined),
\item all vertex and triangle atoms that point to $xx'$ (if $x'$ is defined),
\item all vertex and triangle atoms that point to $yy'$ (if $y'$ is defined),
\item $\eta(x)$,
\item all triangle atoms $\eta(xyw)$ with $w \in V(H) \setminus \{x,y\}$.
\end{itemize}
Let $\Delta$ be the union of the following sets
\begin{itemize}
\item $Em(x) \setminus \{v\}$,
\item $Em(y) \setminus \{v\}$,
\item $Em(x')$ (if $x'$ is defined),
\item $Em(y')$ (if $y'$ is defined).
\end{itemize}
Now assume that $x \neq y'$ and $v \not \in \eta(xy,y)$.
Let $\mathcal{A}$ be the union of the following sets
\begin{itemize}
\item $\{v\}$,
\item $\eta(xy) \cap Safe(G)$,
\item $\eta(xx') \cap Safe(G)$ (if $x'$ is defined),
\item all vertex and triangle atoms that point to $xx'$ (if $x'$ is defined),
\item $\eta(x)$,
\item all triangle atoms $\eta(xyw)$ with $w \in V(H) \setminus \{x,y\}$.
\end{itemize}
Let $\Delta$ be the union of the following sets
\begin{itemize}
\item $Em(x) \setminus \{v\}$,
\item $Em(y) \setminus \{v\}$,
\item $Em(x')$ (if $x'$ is defined).
\end{itemize}
Since $\Delta$ is contained in the union of at most four sets emulating a vertex of $H$, it follows that $\alpha(\Delta) < 4t$.
It follows  from the definition of a strip structure that $\Delta$
separates $\mathcal{A}$ from $V(G) \setminus (\mathcal{A} \cup \Delta)$,
and therefore $\Delta$ meets every leaf path ending at $v$ that
does not start in $\mathcal{A}$.  Let $X=\Delta \cup (\mathcal{A} \cap Z)$.
Since there are at
most three edges $e$ of $H$ such that $\overline{A}$ meets the
set $\eta(e)$, it follows that $|\overline{A} \cap Z| \leq 6$.
We deduce that $\alpha(X) <4t+6$, and $X$ meets every leaf path
that ends at $v$, a contradiction.
\end{proof}

%% file: figures/cooperative_set.tikz
\begin{tikzpicture}
	\begin{pgfonlayer}{nodelayer}
		\node [style=none] (0) at (-2, 0) {};
		\node [style=none] (1) at (5.5, 0) {};
		\node [style=none] (2) at (-2, -1) {};
		\node [style=none] (3) at (5.5, -1) {};
		\node [style=none] (4) at (5.5, -2) {};
		\node [style=none] (5) at (-2, -2) {};
		\node [style=none] (6) at (-2, -3) {};
		\node [style=none] (7) at (5.5, -3) {};
		\node [style=none] (8) at (5.5, -4) {};
		\node [style=none] (9) at (-2, -4) {};
		\node [style=none] (10) at (-3.5, -7) {};
		\node [style=none] (11) at (-2.75, -5.75) {};
		\node [style=none] (12) at (5.5, -5.75) {};
		\node [style=none] (13) at (7, -8) {};
		\node [style=none] (14) at (-1, -8.5) {};
		\node [style=none] (15) at (5, -1) {};
		\node [style=none] (16) at (-1.5, -1) {};
		\node [style=none] (17) at (1.75, -1.5) {$\delta_2(X)$};
		\node [style=none] (18) at (-1.5, -3) {};
		\node [style=none] (19) at (1.75, -3.5) {$\delta_1(X)$};
		\node [style=none] (20) at (5.25, -3) {};
		\node [style=none] (21) at (1.75, 1) {$X\setminus (\delta_1(X) \cup \delta_2(X))$};
		\node [style=none] (22) at (-1, 0) {};
		\node [style=none] (23) at (4.5, 0) {};
		\node [style=none] (24) at (4.5, -2) {};
		\node [style=none] (25) at (-1, -2) {};
		\node [style=none] (26) at (-1, -4) {};
		\node [style=none] (27) at (4.5, -4) {};
		\node [style=none] (28) at (-1.5, -5.25) {};
		\node [style=none] (29) at (5.25, -5.7) {};
		\node [style=none] (30) at (1.25, -6.75) {$G\setminus X$};
	\end{pgfonlayer}
	\begin{pgfonlayer}{edgelayer}
		\draw (9.center) to (6.center);
		\draw (6.center) to (7.center);
		\draw (7.center) to (8.center);
		\draw (8.center) to (9.center);
		\draw (5.center) to (2.center);
		\draw (2.center) to (3.center);
		\draw (4.center) to (3.center);
		\draw (4.center) to (5.center);
		\draw (0.center) to (1.center);
		\draw [bend left=90, looseness=2.00] (0.center) to (1.center);
		\draw [in=165, out=30, looseness=0.75] (11.center) to (12.center);
		\draw [in=60, out=-15, looseness=1.25] (12.center) to (13.center);
		\draw [in=-15, out=-120, looseness=0.50] (13.center) to (14.center);
		\draw [in=-75, out=165] (14.center) to (10.center);
		\draw [in=-150, out=105] (10.center) to (11.center);
		\draw (16.center) to (22.center);
		\draw (23.center) to (15.center);
		\draw (25.center) to (18.center);
		\draw (24.center) to (20.center);
		\draw (26.center) to (28.center);
		\draw (27.center) to (29.center);
	\end{pgfonlayer}
\end{tikzpicture}

%% file: figures/n_cooperative.tikz
\begin{tikzpicture}
	\begin{pgfonlayer}{nodelayer}
		\node [style=none] (0) at (-2, 0) {};
		\node [style=none] (1) at (5.5, 0) {};
		\node [style=none] (2) at (-2, -1) {};
		\node [style=none] (3) at (5.5, -1) {};
		\node [style=none] (4) at (5.5, -2) {};
		\node [style=none] (5) at (-2, -2) {};
		\node [style=none] (6) at (-2, -3) {};
		\node [style=none] (7) at (5.5, -3) {};
		\node [style=none] (8) at (5.5, -4) {};
		\node [style=none] (9) at (-2, -4) {};
		\node [style=none] (15) at (5, -1) {};
		\node [style=none] (16) at (-1.5, -1) {};
		\node [style=none] (17) at (1.75, -1.5) {$\delta_2(X)$};
		\node [style=none] (18) at (-1.5, -3) {};
		\node [style=none] (19) at (1.75, -3.5) {$\delta_1(X)$};
		\node [style=none] (20) at (5.25, -3) {};
		\node [style=none] (21) at (1.75, 1) {$X\setminus (\delta_1(X) \cup \delta_2(X))$};
		\node [style=none] (22) at (-1, 0) {};
		\node [style=none] (23) at (4.5, 0) {};
		\node [style=none] (24) at (4.5, -2) {};
		\node [style=none] (25) at (-1, -2) {};
		\node [style=none] (26) at (-1, -4) {};
		\node [style=none] (27) at (4.5, -4) {};
		\node [style=none] (28) at (-1.5, -5.5) {};
		\node [style=none] (29) at (5.25, -5.45) {};
		\node [style=none] (31) at (4.5, -4) {};
		\node [style=none] (32) at (4.5, -4) {};
		\node [style=none] (33) at (-2, -5.5) {};
		\node [style=none] (34) at (5.5, -5.5) {};
		\node [style=none] (35) at (5.5, -6.5) {};
		\node [style=none] (36) at (-2, -6.5) {};
		\node [style=none] (37) at (-1.5, -5.5) {};
		\node [style=none] (38) at (1.75, -6) {$N(X)$};
		\node [style=none] (39) at (5.25, -5.5) {};
		\node [style=none] (40) at (0.25, -6.5) {};
		\node [style=none] (44) at (-2, -9) {};
		\node [style=none] (45) at (-1.25, -8) {};
		\node [style=none] (46) at (-0.5, -9) {};
		\node [style=none] (47) at (-1.25, -10) {};
		\node [style=none] (48) at (-1.25, -9) {$D_1$};
		\node [style=none] (49) at (-1.85, -8.35) {};
		\node [style=none] (50) at (-0.675, -8.35) {};
		\node [style=none] (51) at (0.5, -9) {};
		\node [style=none] (52) at (1.25, -8) {};
		\node [style=none] (53) at (2, -9) {};
		\node [style=none] (54) at (1.25, -10) {};
		\node [style=none] (55) at (1.25, -9) {$D_2$};
		\node [style=none] (56) at (0.65, -8.35) {};
		\node [style=none] (57) at (1.825, -8.35) {};
		\node [style=none] (58) at (3.5, -9) {};
		\node [style=none] (59) at (4.25, -8) {};
		\node [style=none] (60) at (5, -9) {};
		\node [style=none] (61) at (4.25, -10) {};
		\node [style=none] (62) at (4.25, -9) {$D_3$};
		\node [style=none] (63) at (3.65, -8.35) {};
		\node [style=none] (64) at (4.825, -8.35) {};
		\node [style=none] (65) at (-0.5, -6.5) {};
		\node [style=none] (66) at (2.25, -6.5) {};
		\node [style=none] (67) at (2.75, -6.5) {};
		\node [style=none] (68) at (-3, 0) {};
		\node [style=none] (70) at (6.5, 0) {};
		\node [style=none] (71) at (5.75, -4.5) {};
		\node [style=none] (72) at (1.25, -4.75) {};
		\node [style=none] (74) at (0.1, -8.5) {};
		\node [style=none] (76) at (-3, -8.5) {};
		\node [style=none] (77) at (-3, -3.75) {};
		\node [style=none] (78) at (-2.5, -4.75) {$G'$};
	\end{pgfonlayer}
	\begin{pgfonlayer}{edgelayer}
		\draw (9.center) to (6.center);
		\draw (6.center) to (7.center);
		\draw (7.center) to (8.center);
		\draw (8.center) to (9.center);
		\draw (5.center) to (2.center);
		\draw (2.center) to (3.center);
		\draw (4.center) to (3.center);
		\draw (4.center) to (5.center);
		\draw (0.center) to (1.center);
		\draw [bend left=90, looseness=2.00] (0.center) to (1.center);
		\draw (16.center) to (22.center);
		\draw (23.center) to (15.center);
		\draw (25.center) to (18.center);
		\draw (24.center) to (20.center);
		\draw (26.center) to (28.center);
		\draw (27.center) to (29.center);
		\draw (36.center) to (33.center);
		\draw (33.center) to (34.center);
		\draw (34.center) to (35.center);
		\draw (35.center) to (36.center);
		\draw [in=-180, out=90] (44.center) to (45.center);
		\draw [in=90, out=0] (45.center) to (46.center);
		\draw [in=0, out=-90] (46.center) to (47.center);
		\draw [in=-90, out=-180] (47.center) to (44.center);
		\draw (49.center) to (36.center);
		\draw (50.center) to (40.center);
		\draw [in=-180, out=90] (51.center) to (52.center);
		\draw [in=90, out=0] (52.center) to (53.center);
		\draw [in=0, out=-90] (53.center) to (54.center);
		\draw [in=-90, out=-180] (54.center) to (51.center);
		\draw [in=-180, out=90] (58.center) to (59.center);
		\draw [in=90, out=0] (59.center) to (60.center);
		\draw [in=0, out=-90] (60.center) to (61.center);
		\draw [in=-90, out=-180] (61.center) to (58.center);
		\draw (64.center) to (35.center);
		\draw (63.center) to (67.center);
		\draw (57.center) to (66.center);
		\draw (56.center) to (65.center);
		\draw [style=box edge, in=90, out=90, looseness=1.75] (70.center) to (68.center);
		\draw [style=box edge, in=15, out=-90, looseness=0.75] (70.center) to (71.center);
		\draw [style=box edge, in=0, out=-165, looseness=0.75] (71.center) to (72.center);
		\draw [style=box edge] (76.center) to (77.center);
		\draw [style=box edge] (77.center) to (68.center);
		\draw [style=box edge, in=-90, out=-90, looseness=2.25] (76.center) to (74.center);
		\draw [style=box edge, in=90, out=-180, looseness=1.25] (72.center) to (40.center);
		\draw [style=box edge, in=90, out=-90] (40.center) to (74.center);
	\end{pgfonlayer}
\end{tikzpicture}